\documentclass[11pt, reqno]{amsart}

\usepackage{amsmath, amssymb, amsthm, mathtools, amsfonts, amsbsy, amscd, mathrsfs}
\usepackage{graphicx, xcolor, bm}
\usepackage{multirow, tabularx, booktabs, array, longtable}
\usepackage{enumitem}
\usepackage{tikz-cd}
\usepackage{systeme}
\usepackage{float}
\usepackage{xurl} % Permet la césure des URL longues
\usepackage{hyperref}
\usepackage{caption}
\usepackage{capt-of}
\usepackage[numbers,sort]{natbib}
\newcommand{\emptycomment}[1]{}
\newcommand{\subheading}[1]{\medskip\noindent\textbf{#1}\enspace}
\allowdisplaybreaks
\renewcommand{\arraystretch}{1.0}
\numberwithin{equation}{section}
\newtheorem{defn}{Definition}[section]
\newtheorem{theorem}{Theorem}[section]  
\newtheorem{proposition}{Proposition}[section]  
\newtheorem{lemma}{Lemma}[section] 
  
\newtheorem{remark}{Remark}[section]

\begin{document}
	\title[Non-abelian extensions of Hom-Jacobi-Jordan algebras]{Non-abelian extensions of Hom-Jacobi-Jordan algebras}
	
	\author[N. Saadaoui]{Nejib Saadaoui}

	%%%%%%%%%%%%%%%%%%%%%%%%%%%%%%%%%%%%%%%%%%%%%%%%%%%%%%%%%%%%%%%%%%%%%%%%%%
	\address{Laboratory of Mathematics and Applications LR17ES1,
		Higher Institute of Computer Science and Multimedia Gabes,
		University of Gabès, Tunisia }
	\email{najibsaadaoui@yahoo.fr,  najib.saadaoui@isimg.tn}
	%%%%%%%%%%%%%%%%%%%%%%%%%%%%%%%%%%%
	  
	\maketitle
	%%%%%%%%%%%%%%%%%%%%%%%%%%%%%%%%%%%%%

	\begin{center}
		\small  \today
	\end{center}
	%%%%%%%%%%%%%%%%%%%%%%%%%%%%%%%%%%%%%%%%%%%%%%%%%%%%%%%%%%

	\begin{abstract}
		In this paper, we develop a cohomology theory for Hom-Jacobi-Jordan
		algebras in which $2$-cocycles are pairs $(\rho,\theta)$ rather than
		single cochains, and apply it to the study of extensions. We prove that
		equivalence classes of split extensions are in one-to-one
		correspondence with the second cohomology set $H^2_{\mathrm{na}}(J,V)$.
		Extensions are described explicitly in terms of $2$-cocycles satisfying
		concrete compatibility conditions. We also establish structural results,
		including solvability, nilpotency, and a Fitting-type decomposition
		for nonregular Hom-Jacobi-Jordan
		algebras.
		Finally, we classify low-dimensional Hom-Jacobi-Jordan algebras and
		describe their associated $2$-cocycles.
	\end{abstract}
	\paragraph{Keywords.}
	Hom-Jacobi-Jordan algebra,  non-abelian extensions, 
	cohomology, classification.
	
	\paragraph{MSC (2020).}
	17D30, 17C50, 17A30, 17B61.
	%%%%%%%%%%%%%%%%%%%%%%%%%%%%%%%%%%
	\tableofcontents
	
	\section*{Introduction}
	
	The theory of Hom-type algebras has attracted considerable attention as a natural generalization of classical algebraic structures. In particular, Hom-Lie algebras, introduced by Hartwig, Larsson, and Silvestrov in \cite{HLS}, are nonassociative algebras whose structure is governed by a twisting linear map and a modified Jacobi identity, recovering classical Lie algebras when the twisting map is the identity. Since their introduction, various structural properties of Hom-type algebras have been extensively studied, including representations, cohomology, deformations, solvability, and nilpotency.
	
	Jacobi--Jordan algebras (or mock-Lie algebras) are commutative analogues of Lie algebras, closely related to $\delta$-Jordan--Lie structures \cite{BB2, BCHM, Braiek, Burde-Fialowski2014, Kamia, MMM,Pasha}. Motivated by the development of Hom-type generalizations, Hom-Jacobi-Jordan algebras were introduced as deformations of mock-Lie algebras induced by suitable algebra morphisms \cite{BCHM2,Cyrille,Saadaoui}.
		
	In this paper, we present and study Hom-Jacobi-Jordan algebras with an emphasis on both their structural and cohomological properties. We first establish several fundamental structural results. In particular, we investigate solvability and nilpotency, and analyze the structure of their ideals. Moreover, using a Fitting-type decomposition, we describe the structure of finite-dimensional nonregular  Hom-Jacobi-Jordan algebras.
		
	The second part of the paper is devoted to the cohomology theory and extensions of Hom-Jacobi-Jordan algebras. While non-abelian extensions have been widely studied for various algebraic structures \cite{nab10, nab6, nab13}, their adaptation to the Hom-Jacobi-Jordan setting is not a trivial formal replacement. The general framework of describing extensions via a pair $(\rho, \theta)$ is inspired by analogous theories for Lie and Leibniz algebras \cite{nab10, nab13}; however, the genuinely new aspect lies in the specific compatibility conditions, which are heavily constrained by the \emph{symmetric} nature of the Hom-Jacobi-Jordan bracket and the intertwining of the twisting maps. Furthermore, this work significantly extends our previous study on abelian and quadratic extensions of Hom-Jacobi-Jordan algebras \cite{Saadaoui} by removing the assumption that the kernel is an abelian ideal. In this non-abelian context, the equivalence classes of $2$-cocycles do not form a group but rather a \emph{pointed set}, denoted $H^2_{na}(J, V)$, due to the nonlinear nature of the equivalence relation. We establish a natural bijection between this second cohomology set and the equivalence classes of split extensions.
	
	Finally, we investigate low-dimensional  Hom-Jacobi-Jordan algebras. We provide explicit classifications of decomposable and regular cases in small dimensions (up to dimension $3$), and we fully treat the nonregular cases with nonzero regular part; the only case left for future work is the nonregular one with \(\mathrm{Im}(\alpha_M^3) = 0\). We compute the corresponding cohomology sets, illustrating the rich interaction between structure theory and extension theory in the Hom-setting.
	
	The paper is organized as follows. In Section~\ref{sec1}, we recall basic definitions and establish structural properties. In Section~\ref{sec2}, we develop the second cohomology theory. Section~\ref{sec3} is devoted to extensions and their classification. In the final section, we present low-dimensional classifications and explicit cohomological computations.
	
	\medskip
	\noindent\textbf{Convention.} Unless stated otherwise, $\mathbb{K}$ is a field of characteristic $\neq 2$ and $\neq 3$; in Section~\ref{sec:application}, $\mathbb{K}=\mathbb{C}$.

	\section{Hom-Jacobi-Jordan algebras}\label{sec1}
	In this section, we introduce Hom-Jacobi-Jordan algebras, establish their solvability and nilpotency properties by extending the classical case, and prove that every finite-dimensional nonregular Hom-Jacobi-Jordan algebra $M$ decomposes as a direct sum $M = J \oplus K$, where $J$ is a regular Hom-Jacobi-Jordan algebra and $K$ is an ideal of $M$.
	\begin{defn}\label{def:HJJ}\cite{Cyrille}
		A \emph{Hom-Jacobi-Jordan algebra} (HJJ algebra, for short) is a triple 
		$(J, [\cdot,\cdot], \alpha)$ consisting of a vector space $J$, 
		a symmetric bilinear map
		$[\cdot,\cdot] \colon J \times J \to J$,
		and a linear map $\alpha \colon J \to J$, satisfying
		\begin{align}
			\alpha([x,y]) &= [\alpha(x), \alpha(y)], 
			\label{def:multiplicative}\\
			[\alpha(x), [y,z]] + [\alpha(y), [z,x]] 
			+ [\alpha(z), [x,y]] &= 0,
			\label{def:jacobi}
		\end{align}
		for all $x, y, z \in J$. When $\alpha = \mathrm{id}$, 
		one recovers the classical Jacobi-Jordan algebras.
	\end{defn}	
	%%%%%%%%%%%%%%%%%%%%%%%%%%%%%%%%%%%%%%%%%%%%%%%%%	
	\begin{defn}\label{def:types}
		A Hom-Jacobi-Jordan algebra $(J, [\cdot,\cdot], \alpha)$ 
		is said to be:
		\begin{enumerate}[label=(\arabic*)]
			\item an \emph{abelian Hom-Jacobi-Jordan algebra} if 
			$[\cdot,\cdot]$ vanishes identically on $J$;
			
		\item a \emph{regular Hom-Jacobi-Jordan algebra} if the twisting map $\alpha$ is bijective. Otherwise, $(J,[\cdot,\cdot],\alpha)$ is called \emph{nonregular}.
			
			%\item a \emph{perfect Hom-Jacobi-Jordan algebra} if 
		%	$J^2 := [J,J] = J$. Otherwise, 
			%$(J,[\cdot,\cdot],\alpha)$ is called 
			%\emph{non-perfect}.	
		\end{enumerate}
	\end{defn}
	%%%%%%%%%%%%%%%%%%%%%%%%%%%%%%%%%%%%%%%
	The following result shows that every regular 
	Hom-Jacobi-Jordan algebra gives rise to an 
	underlying classical Jacobi-Jordan algebra.
	\begin{proposition}\label{prop:Jordan}\cite{Cyrille}
		Let $(J, [\cdot,\cdot], \alpha)$ be a regular 
		Hom-Jacobi-Jordan algebra. Then the bracket
		\begin{equation}\label{eq:twisted-bracket}
			[x, y]_{\alpha^{-1}} := \alpha^{-1}([x,y]),
			\quad \text{for all } x, y \in J,
		\end{equation}
		defines a Jacobi-Jordan algebra structure on $J$.
	\end{proposition}
	%%%%%%%%%%%%%%%%%%%%%%%%%%%%%%%%%%%%%%%%%%%%%%%%%%%%%%%%%%	
	\begin{defn}\label{def:homo}
		A \emph{homomorphism of Hom-Jacobi-Jordan algebras}
		\[
		\phi \colon (J,[\cdot,\cdot],\alpha) \longrightarrow (J',[\cdot,\cdot]',\alpha')
		\]
		is a linear map $\phi : J \to J'$ such that
		\begin{align*}
			\phi \circ \alpha &= \alpha' \circ \phi, \\
			\phi([x,y]) &= [\phi(x),\phi(y)]',
		\end{align*}
		for all $x,y \in J$.
		
		The Hom-Jacobi-Jordan algebras $(J,[\cdot,\cdot],\alpha)$ and
		$(J',[\cdot,\cdot]',\alpha')$ are said to be \emph{isomorphic} if there
		exists a bijective homomorphism between them.
	\end{defn}

	%%%%%%%%%%%%%%%%%%%%%%%%%%%%%%%%%%%%%%%%%%%
	\begin{defn}\label{def:ideal}
		Let $(J, [\cdot,\cdot], \alpha)$ be a Hom-Jacobi-Jordan algebra.
		
		\begin{enumerate}[label=(\arabic*)]
			\item A linear subspace $H \subseteq J$ is called a 
			\emph{Hom-subalgebra} of $J$ if
			\[
			[H,H] \subseteq H \quad \text{and} \quad \alpha(H) \subseteq H.
			\]
			
			\item  A linear subspace $I \subseteq J$ is called an \emph{ideal} 
			of $J$ if
			\[
			[I,J] \subseteq I \quad \text{and} \quad \alpha(I) \subseteq  I.
			\]
			
			\item An ideal $I$ is called an \emph{abelian ideal} of $J$ if 
			$[I,I] = \{0\}$, that is, the bracket vanishes identically on $I$.
		\end{enumerate}
	\end{defn}
	%%%%%%%%%%%%%%%%%%%%%%%%%%%%%%%%%%%%%%%%%%%%%%%%%%
	\begin{defn}\label{def:solvable-nilpotent}
		Let $(J, [\cdot,\cdot], \alpha)$ be a Hom-Jacobi-Jordan algebra.
		\begin{enumerate}[label=(\roman*)]
			\item The \emph{derived series} of $J$ is defined by
			\[
			D^0(J) = J, \qquad 
			D^{k+1}(J) = [D^k(J), D^k(J)], \quad k \geq 0.
			\]
			\item The \emph{lower central series} of $J$ is defined by
			\[
			C^0(J) = J, \qquad 
			C^{k+1}(J) = [J, C^k(J)], \quad k \geq 0.
			\]
		\end{enumerate}
		The Hom-Jacobi-Jordan algebra $J$ is said to be \emph{solvable} 
		(resp.\ \emph{nilpotent}) if there exists an integer 
		$k \geq 0$ such that
		\[
		D^{k+1}(J) = \{0\} 
		\quad \bigl(\text{resp.}\ C^{k+1}(J) = \{0\}\bigr).
		\]
	\end{defn}
	
	\begin{remark}\label{rem:inclusion}
		It is easy to verify that $D^k(J) \subseteq C^k(J)$ 
		for all $k \geq 0$.
	\end{remark}
	
	As a direct consequence of Remark~\ref{rem:inclusion}, 
	we obtain the following result.
	
	\begin{proposition}\label{prop:nilpotent-solvable}
		Every nilpotent Hom-Jacobi-Jordan algebra is solvable.
	\end{proposition}
	
	%\begin{proof}
	%	This follows directly from the inclusion 
	%	$D^k(J) \subseteq C^k(J)$ for all $k \geq 0$.
	%\end{proof}
	%%%%%%%%%%%%%%%%%%%%%%%%%%%%%%%%%%%%%%%%%%%%%%%%%%%%%%%%%%%%%%%
	\begin{proposition}\label{prop:subalgebra-ideal}
		For every $k \in \mathbb{N}$, the subspaces $D^k(J)$ and $C^k(J)$
		are Hom-subalgebras of $J$. Moreover, if $\alpha$ is surjective,
		then $D^k(J)$ and $C^k(J)$ are ideals of $J$.
	\end{proposition}
	%===============================
	
	\begin{proof}
		We prove the statements for the derived series; the lower central series is analogous. We proceed by induction on $k$.
		
		\textit{Step 1:} $D^k(J)$ is a Hom-subalgebra. The case $k=0$ is clear. If $D^k(J)$ is a Hom-subalgebra, then $[D^{k+1}(J),D^{k+1}(J)]\subseteq[D^k(J),D^k(J)]=D^{k+1}(J)$, and $\alpha([b,c])=[\alpha(b),\alpha(c)]\in D^{k+1}(J)$ for $b,c\in D^k(J)$, so $D^{k+1}(J)$ is a Hom-subalgebra.
		
		\textit{Step 2:} $D^k(J)$ is an ideal if $\alpha$ is surjective. For $a=[b,c]\in D^{k+1}(J)$ and $y=\alpha(x)\in J$, the Hom-Jacobi identity gives $[y,a]=-[\alpha(b),[c,x]]-[\alpha(c),[x,b]]\in D^{k+1}(J)$ since $D^k(J)$ is an ideal.
	\end{proof}
	%%%%%%%%%%%%%%%%%%%%%%%%%%%%%%%%%%%%%%%%%%%%%%%%%%%%%%%%%%
	\begin{lemma}\label{lem:JacobiNilpotent}\cite{Burde-Fialowski2014}
		Every finite-dimensional Jacobi-Jordan algebra is nilpotent.
	\end{lemma}

	The following theorem extends this result to the Hom-setting.
	\begin{theorem}\label{thm:regular-nilpotent}
		Every finite-dimensional regular  Hom-Jacobi-Jordan algebra is nilpotent.
	\end{theorem}

	\begin{proof}
		Let $(J, [\cdot,\cdot], \alpha)$ be a regular 
		Hom-Jacobi-Jordan algebra. By 
		Proposition~\ref{prop:Jordan}, the bracket
		\[
		[x, y]_{\alpha^{-1}} := \alpha^{-1}([x,y]),
		\quad \text{for all } x, y \in J,
		\]
		defines a classical Jacobi-Jordan algebra structure on $J$.
		By Lemma~\ref{lem:JacobiNilpotent}, this algebra is 
		nilpotent, so there exists $n \in \mathbb{N}$ such that
		\[
		C^n\bigl(J, [\cdot,\cdot]_{\alpha^{-1}}\bigr) = \{0\}.
		\]
		We claim that this implies 
		$C^n(J, [\cdot,\cdot]) = \{0\}$.
		Indeed, one checks by induction that
		\[
		C^k\bigl(J,[\cdot,\cdot]_{\alpha^{-1}}\bigr)
		=
		\alpha^{-k}\bigl(C^k(J,[\cdot,\cdot])\bigr),
		\quad \text{for all } k \geq 0.
		\] 
		Since $\alpha$ is bijective, 
		$\alpha^{-n}(C^n(J,[\cdot,\cdot])) = \{0\}$
		implies $C^n(J,[\cdot,\cdot]) = \{0\}$.
		Therefore, $(J, [\cdot,\cdot], \alpha)$ is nilpotent.
	\end{proof}
\begin{theorem}\label{thm:decomposition-regular}
	Let $(M,[\cdot,\cdot]_M,\alpha_M)$ be a finite-dimensional regular and non-abelian HJJ algebra. Then  $D(M) = [M,M]_M$ is a nontrivial proper ideal of $M$.
	
	Moreover, let $V$ be any nontrivial proper ideal of $M$ admitting an $\alpha_M$-invariant complementary subspace $J$, i.e.,
	$$
	M = V \oplus J, \qquad \alpha_M(J) \subseteq J.
	$$
	Then $J$ is a Hom-Jacobi-Jordan algebra, with $\alpha_J$ the restriction of $\alpha_M$ to $J$ and $[\cdot,\cdot]_J$ the projection of $[\cdot,\cdot]_M$ onto $J\times J$. Furthermore, if $D(M)\subseteq V$, then $J$ is abelian.
\end{theorem}
	\begin{proof}
		Since $M$ is regular, Theorem~\ref{thm:regular-nilpotent} gives that $M$ is nilpotent, hence solvable by Proposition~\ref{prop:nilpotent-solvable}; in particular $D(M)=[M,M]_M$ is a proper ideal of $M$. Let $V$ be an ideal of $M$ admitting an $\alpha_M$-invariant complementary subspace $J$, so that
		\[
		M=V\oplus J, \qquad \alpha_M(J)\subseteq J.
		\]
		Let $\alpha_J:=\alpha_M|_J$, and let $[\cdot,\cdot]_J$ be the projection of $[\cdot,\cdot]_M$ onto $J$ along $V$. Since $\alpha_M(J)\subseteq J$ and $\alpha_M(V)\subseteq V$, projecting the multiplicativity and Jacobi identities of $M$ onto $J$ shows that $(J,[\cdot,\cdot]_J,\alpha_J)$ is an HJJ algebra. Finally, if $D(M)\subseteq V$, then $[x,y]_M\in D(M)\subseteq V$ for all $x,y\in J$, so its projection onto $J$ vanishes; hence $[\cdot,\cdot]_J\equiv 0$, and $J$ is abelian.
	\end{proof}
	
\begin{theorem}\label{thm:decomposition_HJJ}
Every finite-dimensional nonregular Hom-Jacobi-Jordan algebra $(M, [\cdot, \cdot]_M, \alpha_M)$ decomposes as a direct sum of vector spaces
$$ M = J \oplus K, $$
\textit{where $J$ is a regular Hom-Jacobi-Jordan algebra and $K$ is an ideal of $M$.}
\end{theorem}
\begin{proof}
	Let $(M, [\cdot, \cdot]_M, \alpha_M)$ be a finite-dimensional nonregular Hom-Jacobi-Jordan algebra, and set $m = \dim M$. By the Fitting decomposition theorem applied to $\alpha_M$,
	$$ M = J \oplus K, \quad J = \mathrm{Im}(\alpha_M^m), \quad K = \mathrm{ker}(\alpha_M^m). $$
	We verify that $K$ is an ideal of $M$. First, $\alpha_M(K) \subseteq K$ since for any $x \in K = \mathrm{ker}(\alpha_M^m)$,
	$$ \alpha_M^m(\alpha_M(x)) = \alpha_M^{m+1}(x) = \alpha_M(\alpha_M^m(x)) = 0. $$
	Second, for any $x \in K$ and $y \in M$, the multiplicativity of $\alpha_M$ implies that $[x, y]_M \in K$. Hence $K$ is an ideal of $M$.
	
	The restriction $\alpha := \alpha_M|_J$ is an automorphism of $J$. For $z,t\in J$, writing $z=\alpha_M^m(x)$ and $t=\alpha_M^m(y)$ with $x,y\in M$, multiplicativity gives $[z,t]_M=\alpha_M^m([x,y]_M)\in J$, so $[\cdot,\cdot]_M$ restricts to a bracket $[\cdot,\cdot]$ on $J$, and $(J,[\cdot,\cdot],\alpha)$ is a regular Hom-Jacobi-Jordan algebra.
\end{proof}
%===================================================	
\medskip
\noindent\textbf{Notation.}\phantomsection\label{not:JV-decomposition}
Throughout this paper, by the notation $(M, [\cdot,\cdot]_M, \alpha_M)$ we mean a triple consisting of a vector space $M$ endowed with a symmetric bilinear map $[\cdot,\cdot]_M \colon M \times M \to M$ and a linear map $\alpha_M \colon M \to M$. We assume that $M$ admits a direct sum decomposition $M = J \oplus V$ into $\alpha_M$-invariant subspaces such that $[M,V]_M \subseteq V$, and we set $\alpha_J := \alpha_M|_J$ and $\beta := \alpha_M|_V$.

	Since $M = J \oplus V$, the bracket $[\cdot,\cdot]_M$ induces the following bilinear maps:
	\begin{itemize}
		\item $[\cdot,\cdot]_J \colon J \times J \to J$, the $J$-component of $[\cdot,\cdot]_M$ restricted to $J \times J$;
		\item $\theta \colon J \times J \to V$, the $V$-component of $[\cdot,\cdot]_M$ restricted to $J \times J$;
		\item $[\cdot,\cdot]_V \colon V \times V \to V$, the restriction of $[\cdot,\cdot]_M$ to $V \times V$, well defined since $[V,V]_M \subseteq [M,V]_M \subseteq V$.
	\end{itemize}
	
	Moreover, since $[M,V]_M \subseteq V$, we may define
	\[
	\rho \colon J \to \operatorname{End}(V), \qquad \rho(x)(v) := [x,v]_M.
	\]
	
	For all $x, y \in J$ and $u, v \in V$, the bracket $[\cdot,\cdot]_M$ then decomposes as
	\begin{equation}\label{eq:bracket-decomposition}
		[x+u,\,y+v]_M = [x,y]_J + \theta(x,y) + \rho(x)v + \rho(y)u + [u,v]_V =: [x+u,\,y+v]_{(\rho,\theta)}.
	\end{equation}
	%%%%%%%%%%%%%%%%%%%%%%%%%%%%%%%%%%%%%%%%%%%%%%%%%%%%%%%%%%%%%%
	%\section{The second cohomology group of Hom-Jacobi-Jordan algebras}\label{2group}
	%%%%%%%%%%%%%%%%%%%%%%%%%%%%%%%%%%%%%
	\section{Cohomology of Hom-Jacobi-Jordan 
		algebras}\label{sec2}
In this section, we develop a cohomology theory for 
Hom-Jacobi-Jordan algebras, introducing the second cohomology set 
(a group in the abelian case) and the algebraic framework used to 
classify split extensions in Section~\ref{sec3}.
	
	Throughout this section, $(J,[\cdot,\cdot],\alpha)$ and $(V,[\cdot,\cdot]_V,\beta)$ denote two Hom-Jacobi-Jordan algebras.

	\begin{defn}\label{def:kcochain}
		A \emph{$k$-cochain} of $J$ with values in $V$ is a 
		symmetric $k$-linear map $f \in S^k(J,V)$ satisfying 
		the compatibility condition
		\[
		\beta \circ f = f \circ \alpha^{\otimes k}.
		\]
		The space of all $k$-cochains is denoted by
		\begin{equation}
			C^k_{\alpha,\beta}(J,V) 
			= \bigl\{ f \in S^k(J,V) \mid 
			\beta \circ f = f \circ \alpha^{\otimes k} \bigr\}.
		\end{equation}
	\end{defn}
	\begin{remark}\label{rem:kcochain-eigenvalue}
		Let $J=\langle u\rangle$ with $\alpha(u)=au$. A $k$-cochain $f \in C^k_{\alpha,\beta}(J,V)$ is nonzero iff $a^k$ is an eigenvalue of $\beta$, with $f(u,\dots,u)$ as a corresponding eigenvector.
	\end{remark}
	%%%%%%%%%%%%%%%%%%%%%%%%%%%%%%%%%%%%%%%%%%%%%%%%%%%%%%%%%%%%%%%%%%%%%%
	Recall the notation $\rho$, $\theta$, $\beta$ introduced above 
	in the end of Section ~\ref{sec1}
; the following definition 
	characterizes such pairs axiomatically.
	
	%%%%%%%%%%%%%%%%%%%%%%%%%%%%%%%%%%%%%%%%%%%%%%%%%
		\begin{defn}
			A \emph{$2$-cocycle} of $J$ with values in $V$ is a pair $(\rho,\theta)$ where $\rho\colon J\to\operatorname{End}(V)$ is linear, $\theta$ satisfies
			\begin{equation}\label{def:2cohain}
				\theta\in C^2_{\alpha,\beta}(J,V),	 
			\end{equation}
			and for all $x,y,z\in J$, $u,v\in V$:
			\begin{gather}
				\rho(\alpha(x))\circ \beta
				=
				\beta\circ \rho(x), \label{def:rep1}
				\\
				\rho([x,y])\beta(v)
				=
				-\rho(\alpha(x))\rho(y)v
				-\rho(\alpha(y))\rho(x)v
				-[\theta(x,y),\beta(v)]_V,
				\label{def:rep2}
				\\
				\theta(\alpha(x),[y,z])
				+\theta(\alpha(y),[x,z])
				+\theta(\alpha(z),[x,y])
				\nonumber\\
				+\rho(\alpha(x))\theta(y,z)
				+\rho(\alpha(y))\theta(x,z)
				+\rho(\alpha(z))\theta(x,y)
				=
				0,
				\label{def:cocycle1}
				\\
				\rho(\alpha(x))[u,v]_V
				=
				-[\beta(u),\rho(x)v]_V
				-[\beta(v),\rho(x)u]_V.
				\label{def:rep3}
			\end{gather}
			The set of all $2$-cocycles of $J$ with values in $V$ is denoted by $Z^2_{\mathrm{na}}(J,V)$.
		\end{defn}
	%%%%%%%%%%%%%%%%%%%%%%%%%%%%%%%%%%%%%%%%%%%%%%%%%%%%%%%%%%%%%%
	\begin{theorem}\label{thm:structureMAY24}
		With the notation in in the end of Section ~\ref{sec1}, 
		$(M, [\cdot,\cdot]_M, \alpha_M)$ is a 
		Hom-Jacobi-Jordan algebra if and only if:
		\begin{enumerate}[label=(\arabic*)]
			\item $(J, [\cdot,\cdot]_J, \alpha_J)$ and $(V, [\cdot,\cdot]_V, \beta)$ are 
			Hom-Jacobi-Jordan algebras;
			\item $(\rho, \theta)$ is a $2$-cocycle of $J$ 
			with coefficients in $V$.
		\end{enumerate}
	\end{theorem}
	\begin{proof}
		Let $x, y, z \in J$ and $u, v, w \in V$.
		
		Suppose that $(M, [\cdot,\cdot]_M, \alpha_M)$ is a Hom-Jacobi-Jordan algebra.
		From the multiplicativity condition
		\[
		\alpha_M([x+u,\, y+v]_M) = [\alpha(x)+\beta(u),\, \alpha(y)+\beta(v)]_M,
		\]
		and splitting into $J$- and $V$-components, by taking respectively 
		$(x,y)$,  $(x,v)$, $(u,v)$, we obtain the following four identities:
		\[
		\alpha([x,y]) = [\alpha(x),\alpha(y)], \quad \theta(\alpha(x),\alpha(y))=\beta(\theta(x,y)), \quad
		\beta(\rho(x)v) = \rho(\alpha(x))\beta(v), \quad
		\beta([u,v]_{V}) = [\beta(u),\beta(v)]_V.
		\]
				Applying the Hom-Jacobi identity to $x, y, z \in J$ and using the decomposition $M = J \oplus V$, we obtain
		\begin{align*}
			&\underbrace{[\alpha(x),[y,z]]
				+ [\alpha(y),[x,z]]
				+ [\alpha(z),[x,y]]}_{\in\, J} \\
			&+ \underbrace{\theta(\alpha(x),[y,z])
				+ \theta(\alpha(y),[x,z])
				+ \theta(\alpha(z),[x,y])}_{\in\, V} \\
			&+ \underbrace{\rho(\alpha(x))\theta(y,z)
				+ \rho(\alpha(y))\theta(x,z)
				+ \rho(\alpha(z))\theta(x,y)}_{\in\, V}
			= 0.
		\end{align*}
		Since $J \cap V = \{0\}$, the $J$-component and the $V$-component vanish separately.
		
		Applying the Hom-Jacobi identity to $x, y \in J$ and $v \in V$ yields
		\[
		\rho([x,y])\beta(v)
		+ \rho(\alpha(x))\rho(y)v
		+ \rho(\alpha(y))\rho(x)v
		+ [\theta(x,y),\beta(v)]_V = 0.
		\]
		
		Applying the Hom-Jacobi identity to $x \in J$ and $u, v \in V$ yields
		\[
		\rho(\alpha(x))[u,v]_V 
		+ [\beta(u), \rho(x)v]_V 
		+ [\beta(v), \rho(x)u]_V = 0.
		\]	
		Finally, applying the Hom-Jacobi identity to $u, v, w \in V$ gives
		\[
		[\beta(u),[v,w]_V]_V
		+ [\beta(v),[w,u]_V]_V
		+ [\beta(w),[u,v]_V]_V
		= 0.
		\]
	This completes the verification of the necessary conditions, establishing (1) and (2).
	
	The converse follows by symmetrically reversing each of the four 
	identities above, together with the Hom-Jacobi computations, to 
	reconstruct multiplicativity and the Hom-Jacobi identity on $M$.	
	\end{proof}
	
	%%%%%%%%%%%%%%%%%%%%%%%%%%%%%%%%%%%%%%%%%%%%%%%%%%%%%%%%%%%%%%%%%%%%%%%%%%%%%%%%
	\subsection{The abelian case}\label{subsec:abelian}
	
	In this subsection, we assume that $[\cdot,\cdot]_V = 0$,
	that is, $V$ is an abelian Hom-Jacobi-Jordan algebra.
	Under this assumption, condition~\eqref{def:rep3} is 
	automatically satisfied, and the remaining conditions 
	of a $2$-cocycle reduce to:
	\begin{align}
		\rho(\alpha(x)) \circ \beta
		&= \beta \circ \rho(x),
		\label{rep1-ab}\\
		\rho([x,y])(\beta(v))
		&= -\rho(\alpha(x))\rho(y)v
		- \rho(\alpha(y))\rho(x)v,
		\label{rep2-ab}\\
		\theta(\alpha(x),[y,z])
		+ \theta(\alpha(y),[x,z])
		+ \theta(\alpha(z),[x,y])&
		\nonumber\\
		+ \rho(\alpha(x))\theta(y,z)
		+ \rho(\alpha(y))\theta(x,z)
		+ \rho(\alpha(z))\theta(x,y)
		&= 0.
		\label{cocycle-ab}
	\end{align}
	
	We now define the coboundary operators $d^1$ and $d^2$ 
	and introduce the second cohomology group 
	$H^2_{\alpha,\beta}(J,V)$ in this setting.
	%%%%%%%%%%%%%%%%%%%%%%%%%%%%%%%%%%%%%%%%%%%%%%%%%
	
	\begin{defn}\label{def:d1}
		The \emph{$1$-coboundary operator} associated with 
		the Hom-Jacobi-Jordan algebra $J$ is the map
		\[
		d^1 \colon
		C^1_{\alpha,\beta}(J,V)
		\longrightarrow
		C^2_{\alpha,\beta}(J,V),
		\qquad
		f \longmapsto d^1(f),
		\]
		defined by
		\begin{equation}\label{adjoint1}
			d^1(f)(x,y)
			= f([x,y])
			- \rho(x)f(y)
			- \rho(y)f(x),
		\end{equation}
		for all $x, y \in J$.
	\end{defn}
	
	\begin{defn}\label{def:d2}
		The \emph{$2$-coboundary operator} associated with 
		the Hom-Jacobi-Jordan algebra $J$ is the map
		\[
		d^2 \colon
		C^2_{\alpha,\beta}(J,V)
		\longrightarrow
		C^3_{\alpha,\beta}(J,V),
		\qquad
		f \longmapsto d^2(f),
		\]
		defined by
		\begin{align}\label{adjoint2}
			d^2(f)(x,y,z)
			&= f(\alpha(x),[y,z])
			+ f(\alpha(y),[x,z])
			+ f(\alpha(z),[x,y])
			\nonumber\\
			&\quad
			+ \rho(\alpha(x))f(y,z)
			+ \rho(\alpha(y))f(x,z)
			+ \rho(\alpha(z))f(x,y),
		\end{align}
		for all $x, y, z \in J$.
	\end{defn}
	\begin{remark}\label{rem:cocycle-operator}
		Thus, the explicit cocycle condition~\eqref{cocycle-ab} is equivalent to 
		$d^2(\theta) = 0$, i.e., $\theta \in \ker d^2$.
	\end{remark}
	
	%%%%%%%%%%%%%%%%%%%%%%%%%%%%%%%%%%%%%%%%%
	
	\begin{lemma}\label{lem:d1-welldefined}
	For every $f \in C^1_{\alpha,\beta}(J,V)$, we have
	\[
	\beta \circ d^1(f) = d^1(f) \circ \alpha^{\otimes 2}.
	\]
	In particular, the operator
	\(d^1:C^1_{\alpha,\beta}(J,V)\to C^2_{\alpha,\beta}(J,V)\)
	is well defined.		
	\end{lemma}
	
	\begin{proof}
	This follows directly from the compatibility condition
	$f \circ \alpha = \beta \circ f$, condition~\eqref{def:rep1}, 
	and the multiplicativity of $\alpha$.	
	\end{proof}

	\begin{theorem}\label{thm:d2d1=0}
		The coboundary operators satisfy $d^2 \circ d^1 = 0$.
	\end{theorem}
	\begin{proof}
		Let $f \in C^1_{\alpha,\beta}(J,V)$. A direct computation yields $d^2(d^1f)(x,y,z) = A(x,y,z) + B(x,y,z)$, where
		\[
		A(x,y,z) = f([\alpha(x),[y,z]]) + f([\alpha(y),[x,z]]) + f([\alpha(z),[x,y]])
		\]
		and $B(x,y,z)$ collects all terms involving $\rho$. By the Hom-Jacobi identity~\eqref{def:jacobi}, $A(x,y,z) = 0$. 
		
		Using $f\circ\alpha=\beta\circ f$ and condition~\eqref{rep2-ab}, we obtain
		\[
		\rho([y,z])f(\alpha(x)) = -\rho(\alpha(y))\rho(z)f(x) - \rho(\alpha(z))\rho(y)f(x).
		\]
		Substituting this and its cyclic permutations into $B(x,y,z)$, all terms cancel pairwise, yielding $B(x,y,z)=0$. Hence, $d^2\circ d^1=0$.
	\end{proof}
	
	%%%%%%%%%%%%%%%%%%%%%%%%%%%%%%%%
	\begin{defn}\label{def:H2}
		The \emph{second cohomology group} of $J$ with 
		coefficients in $V$ is defined as the quotient
		\[
		H^2_{\alpha,\beta}(J,V)
		=
		Z^2_{\alpha,\beta}(J,V)
		\,/\,
		B^2_{\alpha,\beta}(J,V),
		\]
		where
		\[
		Z^2_{\alpha,\beta}(J,V) = \ker d^2
		\]
		is the space of \emph{$2$-cocycles}, and
		\[
		B^2_{\alpha,\beta}(J,V) = \operatorname{Im} d^1
		\]
		is the space of \emph{$2$-coboundaries}.
		Note that $B^2_{\alpha,\beta}(J,V) \subseteq 
		Z^2_{\alpha,\beta}(J,V)$ by 
		Theorem~\ref{thm:d2d1=0}, so the quotient 
		is well defined.
	\end{defn}
	%%%%%%%%%%%%%%%%%%%%%%%%%%%%%%%%%%%%
	\subsection{The non-abelian case}\label{subsec:nonabelian}
	
	In this subsection, we consider the general case 
	where $[\cdot,\cdot]_V$ is not necessarily zero.
	
	Given a $2$-cocycle $(\rho,\theta)$ of $J$ with values in $V$, we write
	$[\cdot,\cdot]_{(\rho,\theta)}$ for the bracket on $J\oplus V$ defined by
	\[
	[x+u,y+v]_{(\rho,\theta)} := [x,y] + \theta(x,y) + \rho(x)v + \rho(y)u + [u,v]_V,
	\]
	which, by Theorem~\ref{thm:structureMAY24}, makes 
	$(J\oplus V,[\cdot,\cdot]_{(\rho,\theta)},\alpha+\beta)$ a 
	Hom-Jacobi-Jordan algebra.
	\begin{defn}\label{def:equivalent-cocycles}
		Two $2$-cocycles $(\rho, \theta)$ and $(\rho', \theta')$ are said to be \emph{equivalent} if there exists a $1$-cochain
		$h \in C^1_{\alpha,\beta}(J,V)$ such that the map
		\[
		\Phi_{h}:(J\oplus V,[\cdot ,\cdot ]_{(\rho, \theta)},\alpha+\beta) \longrightarrow (J\oplus V,[\cdot ,\cdot ]_{(\rho', \theta')},\alpha+\beta)
		\]
		defined by
		\[
		\Phi_{h}(x+v)=x+h(x)+v
		\]
		is an isomorphism of Hom-Jacobi-Jordan algebras. 	In this case, we write $(\rho,\theta) \sim (\rho',\theta')$.
	\end{defn}
	\begin{proposition}\label{prop:equivalence-relation}
		The relation $\sim$ is an equivalence relation on the 
		set of $2$-cocycles. We denote by 
		$H^2_{\mathrm{na}}(J,V)$ the set of equivalence 
		classes of $2$-cocycles under this relation.
	\end{proposition}
	
	\begin{proof}
		\textit{Reflexivity}: Take $h = 0$. The isomorphism condition on $\Phi_h$ is trivially 
		satisfied, so $(\rho,\theta) \sim (\rho,\theta)$.
		
		\textit{Symmetry}: If $(\rho,\theta) \sim (\rho',\theta')$ 
		via $h \in C^1_{\alpha,\beta}(J,V)$, then 
		$(\rho',\theta') \sim (\rho,\theta)$ via $-h$, 
		since $-h \in C^1_{\alpha,\beta}(J,V)$ and the 
		conditions are satisfied with opposite sign.
		
		\textit{Transitivity}: If $(\rho,\theta) \sim (\rho',\theta')$ via $h \in C^1_{\alpha,\beta}(J,V)$ 
		and $(\rho',\theta') \sim (\rho'',\theta'')$ via $h' \in C^1_{\alpha,\beta}(J,V)$, 
		then $(\rho,\theta) \sim (\rho'',\theta'')$ via $h + h'$. 
		This follows by direct substitution, noting that $h + h' \in C^1_{\alpha,\beta}(J,V)$ 
		by linearity. 
	\end{proof}
	%%%%%%%%%%%%%%%%%%%%%%%%%%%%%%%%%%%%%%%%%%%%%%%%%%%%%%
	\begin{proposition}\label{prop:cocycle-equivalence-formulas}
		Let $(\rho, \theta)$ and $(\rho', \theta')$ be two 
		equivalent $2$-cocycles. Then there exists a $1$-cochain 
		$h \in C^1_{\alpha,\beta}(J,V)$ such that	
		\begin{align}
			\rho'(x)v
			&= \rho(x)v - [h(x), v]_V,
			\label{eq:equivalent-cocycle1}\\
			\theta'(x,y)
			&= \theta(x,y)
			+ h([x,y])
			- \rho(x)h(y)
			- \rho(y)h(x)
			+ [h(x), h(y)]_V,
			\label{eq:equivalent-cocycle2}
		\end{align}
		for all $x, y \in J$ and $v \in V$.
	\end{proposition}
	\begin{proof}
		Applying the homomorphism condition $\Phi_h([x,v]_{(\rho,\theta)}) = [\Phi_h(x), \Phi_h(v)]_{(\rho',\theta')}$ 
		for $x \in J$ and $v \in V$ yields \eqref{eq:equivalent-cocycle1}. Similarly, applying 
		$\Phi_h([x,y]_{(\rho,\theta)}) = [\Phi_h(x), \Phi_h(y)]_{(\rho',\theta')}$ for $x,y \in J$, and substituting \eqref{eq:equivalent-cocycle1} together with the symmetry of $[\cdot,\cdot]_V$, gives \eqref{eq:equivalent-cocycle2}.
	\end{proof}
	\begin{remark}
		Unlike the abelian case, the quadratic term $[h(x),h(y)]_V$ 
		in~\eqref{eq:equivalent-cocycle2} prevents $H^2_{\mathrm{na}}(J,V)$ 
		from carrying a natural group structure: it is only a pointed set, 
		with base point the class of $(0,0)$.
	\end{remark}
	\begin{remark}\label{rem:abelian-reduction}
		In the abelian case $[\cdot,\cdot]_V = 0$, 
		condition~\eqref{eq:equivalent-cocycle1} reduces to 
		$\rho' = \rho$, and 
		condition~\eqref{eq:equivalent-cocycle2} reduces to
		\[
		\theta'(x,y) = \theta(x,y) + d^1(h)(x,y),
		\]
		which is the classical coboundary relation.
	\end{remark}
	%%%%%%%%%%%%%%%%%%%%%%%%%%%%%%%%%%%%%
	%%%%%%%%%%%%%%%%%%%%%%%%
	%%%%%%%%%%%%%%%%%%%%%%%%%%%%%%%%%%%%%%%%%%%%%%%%%%%%%%%%%%%%%%%%%%%%%%%%%%%%%%%%%%%%%%%%%%%%%%%%%%%%%%%%
	\section{Extensions of Hom-Jacobi-Jordan 
		algebras}\label{sec3}
%%%%%%%%%%%%%%%%%%%%%%%%%%%%%%%%%%%%%%%%%
%%%%%%%%%%%%%%%%%%%%%%%%%%%%%%%%%
%%%%%%%%%%%%%%%%%%%%%%%%%%%%	
In this section, we study extensions of Hom-Jacobi-Jordan algebras 
and relate them to the cohomology theory of Section~\ref{sec2}, 
culminating in Theorem~\ref{thm:ext-cohomology}: a natural bijection 
between $H^2_{\mathrm{na}}(J,V)$ and the set $\mathrm{Ext}(J,V)$ of 
equivalence classes of split extensions of $J$ by $V$.		
	\begin{defn}\label{def:extension}
		Let $(J, [\cdot,\cdot], \alpha)$, 
		$(V, [\cdot,\cdot]_V, \beta)$, and 
		$(M, [\cdot,\cdot]_M, \alpha_M)$ be 
		Hom-Jacobi-Jordan algebras, and let
		\[
		i \colon V \to M, \qquad \pi \colon M \to J
		\]
		be morphisms of Hom-Jacobi-Jordan algebras.
		The sequence
		\[
		0
		\longrightarrow
		(V, [\cdot,\cdot]_V, \beta)
		\stackrel{i}{\longrightarrow}
		(M, [\cdot,\cdot]_M, \alpha_M)
		\stackrel{\pi}{\longrightarrow}
		(J, [\cdot,\cdot], \alpha)
		\longrightarrow
		0
		\]
		is called a \emph{short exact sequence} if 
		$i$ is injective, $\pi$ is surjective, and
		$\operatorname{Im}(i) = \ker(\pi)$.
		The structure maps are required to satisfy
		\begin{equation}\label{extension9}
			\alpha_M \circ i = i \circ \beta,
			\qquad
			\alpha \circ \pi = \pi \circ \alpha_M.
		\end{equation}
		In this case, $(M, [\cdot,\cdot]_M, i,\pi)$ 
		is called an \emph{extension} of 
		$(J, [\cdot,\cdot], \alpha)$ by 
		$(V, [\cdot,\cdot]_V, \beta)$.
		
		The extension is called \emph{split} if there exists a linear map 
		$s \colon J \to M$, called a \emph{Hom-section}, satisfying
		\[
		\pi \circ s = \operatorname{id}_J \qquad \text{and} \qquad \alpha_M \circ s = s \circ \alpha.
		\]
	\end{defn}

	%%%%%%%%%%%%%%%%%%%%%%%%%%%%%%%%%%%%%%%%%%%%
	\begin{remark}\label{rem:split-invariant}
		The existence of a Hom-section $s \colon J \to M$ is equivalent to the existence of an 
		$\alpha_M$-invariant subspace $X \subseteq M$ such that $M = i(V) \oplus X$. 
		Indeed, if $s$ is a Hom-section, then $X = s(J)$ is clearly $\alpha_M$-invariant and 
		$M = i(V) \oplus s(J)$. Conversely, if such an $X$ exists, the restriction 
		$\pi|_X \colon X \to J$ is a linear isomorphism. Defining $s := (\pi|_X)^{-1}$, 
		the morphism property of $\pi$ together with the invariance $\alpha_M(X) \subseteq X$ 
		implies $\alpha_M \circ s = s \circ \alpha$, so $s$ is a Hom-section.
	\end{remark}
	%%%%%%%%%%%%%%%%%%%%%%%%%%
	\begin{proposition}\label{prop:split-extension}
		Let $(J, [\cdot,\cdot], \alpha)$ and 
		$(V, [\cdot,\cdot]_V, \beta)$ be Hom-Jacobi-Jordan 
		algebras. The sequence
		\[
	\mathcal{E}	\colon 0 
		\longrightarrow 
		(V, [\cdot,\cdot]_V, \beta)
		\stackrel{i_0}{\longrightarrow}
		(J \oplus V, [\cdot,\cdot]_{J\oplus V}, \alpha+\beta)
		\stackrel{\pi_0}{\longrightarrow}
		(J, [\cdot,\cdot], \alpha)
		\longrightarrow 0,
		\]
		where $i_0$ and $\pi_0$ are the natural inclusion 
		and projection maps, defines a split extension of 
		$J$ by $V$ if and only if
		\[
		[x+u, y+v]_{J\oplus V}
		= [x,y] + \theta(x,y) 
		+ \rho(x)v + \rho(y)u + [u,v]_V,
		\]
		for all $x, y \in J$ and $u, v \in V$,
		where $(\rho, \theta)$ is a $2$-cocycle of $J$ 
		with values in $V$.
		
	\end{proposition}
	
	\begin{proof}
	The direct implication is straightforward. Conversely, 
	assume the bracket has the stated form for a $2$-cocycle 
	$(\rho,\theta)$ of $J$ with values in $V$; by 
	Theorem~\ref{thm:structureMAY24}, $(J \oplus V, 
	[\cdot,\cdot]_{J\oplus V}, \alpha+\beta)$ is then a 
	Hom-Jacobi-Jordan algebra. It is straightforward to check 
	that $i_0$ and $\pi_0$ satisfy the compatibility 
	conditions~\eqref{extension9}, and that
	\[
	J \oplus V = i_0(V) \oplus s(J), \qquad s(x)=(x,0),
	\]
	with $s(J)$ an $(\alpha+\beta)$-invariant complement to 
	$i_0(V)$; by Remark~\ref{rem:split-invariant}, $s$ is 
	accordingly a Hom-section of $\pi_0$. This shows that $E$ 
	is a split extension of $J$ by $V$, completing the proof.
\end{proof}

	\noindent We refer to $	\mathcal{E}_0$ as the \emph{standard split extension} of $J$ by $V$ associated with the $2$-cocycle $(\rho,\theta)$, 
	and we denote the corresponding bracket on $J \oplus V$ by $[\cdot,\cdot]_{(\rho,\theta)}$.
	%%%%%%%%%%%%%%%%%%%%
	
	We now introduce the notion of equivalence between 
	two extensions of $J$ by $V$.
	
	\begin{defn}\label{def:equivalent-extensions}
		Two extensions
		\[
		0 \longrightarrow 
		(V, [\cdot,\cdot]_V, \beta)
		\stackrel{i_1}{\longrightarrow}
		(M_1, [\cdot,\cdot]_{M_1}, \alpha_{M_1})
		\stackrel{\pi_1}{\longrightarrow}
		(J, [\cdot,\cdot], \alpha)
		\longrightarrow 0
		\]
		and
		\[
		0 \longrightarrow 
		(V, [\cdot,\cdot]_V, \beta)
		\stackrel{i_2}{\longrightarrow}
		(M_2, [\cdot,\cdot]_{M_2}, \alpha_{M_2})
		\stackrel{\pi_2}{\longrightarrow}
		(J, [\cdot,\cdot], \alpha)
		\longrightarrow 0
		\]
		are said to be \emph{equivalent} if there exists a 
		Hom-Jacobi-Jordan algebra isomorphism
		\[
		\Phi \colon 
		(M_1, [\cdot,\cdot]_{M_1}, \alpha_{M_1})
		\longrightarrow
		(M_2, [\cdot,\cdot]_{M_2}, \alpha_{M_2})
		\]
		such that the following diagram commutes:
		\[
		\begin{tikzcd}
			0 \arrow[r]
			& (V,[\cdot,\cdot]_V,\beta)
			\arrow[d,"\mathrm{id}_V"']
			\arrow[r,"i_1"]
			& (M_1,[\cdot,\cdot]_{M_1},\alpha_{M_1})
			\arrow[d,"\Phi"]
			\arrow[r,"\pi_1"]
			& (J,[\cdot,\cdot],\alpha)
			\arrow[d,"\mathrm{id}_J"]
			\arrow[r]
			& 0
			\\
			0 \arrow[r]
			& (V,[\cdot,\cdot]_V,\beta)
			\arrow[r,"i_2"]
			& (M_2,[\cdot,\cdot]_{M_2},\alpha_{M_2})
			\arrow[r,"\pi_2"]
			& (J,[\cdot,\cdot],\alpha)
			\arrow[r]
			& 0
		\end{tikzcd}
		\]
		that is,
		\[
		\Phi \circ i_1 = i_2
		\qquad \text{and} \qquad
		\pi_2 \circ \Phi = \pi_1.
		\]
		We denote by $\mathrm{Ext}(J,V)$ the set of 
		equivalence classes of split extensions of 
		$(J,[\cdot,\cdot],\alpha)$ by 
		$(V,[\cdot,\cdot]_V,\beta)$.
	\end{defn}
	%%%%%%%%%%%%%%%%%%%%%%%%%%%%%%%%%%%%%%%%%%%%%%%
	\begin{theorem}\label{thm:every-extension-standard}
		Let
		\[
			\mathcal{E} \colon 0 
		\longrightarrow 
		(V, [\cdot,\cdot]_V, \beta)
		\stackrel{i}{\longrightarrow}
		(M, [\cdot,\cdot]_M, \alpha_M)
		\stackrel{\pi}{\longrightarrow}
		(J, [\cdot,\cdot], \alpha)
		\longrightarrow 0
		\]
		be a split extension of $J$ by $V$. Then there 
		exists a $2$-cocycle $(\rho, \theta) \in 
		Z^2_{\mathrm{na}}(J,V)$ such that the bracket
		\[
		[x+u, y+v]_{(\rho,\theta)}
		= [x,y] + \theta(x,y) 
		+ \rho(x)v + \rho(y)u + [u,v]_V
		\]
		defines a Hom-Jacobi-Jordan algebra structure 
		on $J \oplus V$, and the standard split extension
		\[
			\mathcal{E}_0 \colon 0 
		\longrightarrow 
		(V, [\cdot,\cdot]_V, \beta)
		\stackrel{i_0}{\longrightarrow}
		(J \oplus V, [\cdot,\cdot]_{(\rho,\theta)}, \alpha+\beta)
		\stackrel{\pi_0}{\longrightarrow}
		(J, [\cdot,\cdot], \alpha)
		\longrightarrow 0
		\]
		is equivalent to $E$.
	\end{theorem}
\begin{proof}
	By Remark~\ref{rem:split-invariant}, there exists a Hom-section $s\colon J\to M$ with $M=i(V)\oplus s(J)$ and $\alpha_M\circ s=s\circ\alpha$. Define $\Phi\colon J\oplus V\to M$ by $\Phi(x,v)=s(x)+i(v)$. Then $\Phi$ is an isomorphism satisfying $\Phi\circ i_0=i$, $\pi\circ\Phi=\pi_0$, and $\Phi\circ(\alpha+\beta)=\alpha_M\circ\Phi$. By Theorem~\ref{thm:structureMAY24}, the bracket on $J\oplus V$ transported via $\Phi$ has the form $[(x,u),(y,v)]=[x,y]+\theta(x,y)+\rho(x)v+\rho(y)u+[u,v]_V$ for some $2$-cocycle $(\rho,\theta)$, so $E_0$ is equivalent to $E$.
\end{proof}	
	%%%%%%%%%%%%%%%%%%%%%%%%%%%%%%%%
	\begin{proposition}\label{prop:equivalent-extensions}
		Let
		\[
			\mathcal{E}_0 \colon 0 
		\longrightarrow 
		(V, [\cdot,\cdot]_V, \beta)
		\stackrel{i_0}{\longrightarrow}
		(J \oplus V, [\cdot,\cdot]_{(\rho,\theta)}, \alpha+\beta)
		\stackrel{\pi_0}{\longrightarrow}
		(J, [\cdot,\cdot], \alpha)
		\longrightarrow 0
		\]
		and
		\[
			\mathcal{E}'_0 \colon 0 
		\longrightarrow 
		(V, [\cdot,\cdot]_V, \beta)
		\stackrel{i_0}{\longrightarrow}
		(J \oplus V, [\cdot,\cdot]_{(\rho',\theta')}, \alpha+\beta)
		\stackrel{\pi_0}{\longrightarrow}
		(J, [\cdot,\cdot], \alpha)
		\longrightarrow 0
		\]
		be two standard split extensions associated with 
		the $2$-cocycles $(\rho,\theta)$ and $(\rho',\theta')$, 
		respectively.
		Then $\mathcal{E}_0$ and $\mathcal{E}'_0$ are equivalent if and only if 
		there exists a $1$-cochain 
		$h \in C^1_{\alpha,\beta}(J,V)$ such that
		\begin{align}
			\rho'(x)v
			&= \rho(x)v -[h(x),v]_V,
			\label{eq:ext-equiv1}\\
			\theta'(x,y)
			&= \theta(x,y) + h([x,y])
			- \rho(x)h(y) - \rho(y)h(x)
			+ [h(x),h(y)]_V.
			\label{eq:ext-equiv2}
		\end{align}
		That is, $(\rho,\theta) \sim (\rho',\theta')$ in the sense of Definition~\ref{def:equivalent-cocycles}.
		
	\end{proposition}
	
	\begin{proof}
		Suppose first that $\mathcal{E}_0$ and $\mathcal{E}'_0$ are equivalent.
		Then there exists a Hom-Jacobi-Jordan algebra 
		isomorphism
		\[
		\Phi \colon 
		(J \oplus V, [\cdot,\cdot]_{(\rho,\theta)}, \alpha+\beta)
		\longrightarrow
		(J \oplus V, [\cdot,\cdot]_{(\rho',\theta')}, \alpha+\beta)
		\]
		satisfying $\Phi \circ i_0 = i_0$ and 
		$\pi_0 \circ \Phi = \pi_0$.
		From these two conditions, one deduces that $\Phi$ 
		has the form
		\[
		\Phi(x+v) = x + h(x) + v,
		\]
		for some linear map $h \colon J \to V$.
		The compatibility condition
		$\Phi \circ (\alpha+\beta) = (\alpha+\beta) \circ \Phi$
		gives
		\[
		h \circ \alpha = \beta \circ h,
		\]
		hence $h \in C^1_{\alpha,\beta}(J,V)$.

	The derivation of \eqref{eq:ext-equiv1} and \eqref{eq:ext-equiv2} 
	from the isomorphism condition on $\Phi$ is then identical to 
	the proof of Proposition~\ref{prop:cocycle-equivalence-formulas}, 
	applied respectively to $x \in J,\ v \in V$ and to $x, y \in J$.	
	
		Conversely, let $h \in C^1_{\alpha,\beta}(J,V)$ 
		satisfy~\eqref{eq:ext-equiv1} 
		and~\eqref{eq:ext-equiv2}. Define
		\[
		\Phi \colon J \oplus V \to J \oplus V,
		\qquad \Phi(x+v) = x + h(x) + v.
		\]
		A direct computation shows that $\Phi$ is an 
		isomorphism of Hom-Jacobi-Jordan algebras 
		satisfying $\Phi \circ i_0 = i_0$ and 
		$\pi_0 \circ \Phi = \pi_0$.
		Therefore, $E_0$ and $E'_0$ are equivalent.
	\end{proof}
	%%%%%%%%%%%%%%%%%%%%%%%%%%%%
	Before stating the main classification result, we summarize the three-step correspondence underlying it.
	\begin{enumerate}
		\item \textbf{Cocycle $\to$ extension:} For $(\rho,\theta)\in Z^2_{\mathrm{na}}(J,V)$, Theorem~\ref{thm:structureMAY24} and Proposition~\ref{prop:split-extension} show that
		\[
		[x+u,y+v]_{(\rho,\theta)} = [x,y] + \theta(x,y) + \rho(x)v + \rho(y)u + [u,v]_V
		\]
		defines a standard split extension $\mathcal{E}_{(\rho,\theta)}$ of $J$ by $V$.
		\item \textbf{Equivalence:} By Proposition~\ref{prop:equivalent-extensions}, $\mathcal{E}_{(\rho,\theta)} \sim \mathcal{E}_{(\rho',\theta')}$ iff the cocycles are equivalent (Definition~\ref{def:equivalent-cocycles}) via some $h \in C^1_{\alpha,\beta}(J,V)$.
		\item \textbf{Every extension is standard:} By Theorem~\ref{thm:every-extension-standard}, every split extension of $J$ by $V$ is equivalent to some $\mathcal{E}_{(\rho,\theta)}$.
	\end{enumerate}
	Together, these steps yield a bijection $H^2_{\mathrm{na}}(J,V) \leftrightarrow \mathrm{Ext}(J,V)$.
	Recall that $\mathrm{Ext}(J,V)$ denotes the set 
	of equivalence classes of split extensions of 
	$(J, [\cdot,\cdot], \alpha)$ by 
	$(V, [\cdot,\cdot]_V, \beta)$, as introduced in 
	Definition~\ref{def:equivalent-extensions}.
	For an extension $\mathcal{E}$, we denote its 
	equivalence class by $[\mathcal{E}]$, and for a 
	$2$-cocycle $c = (\rho,\theta)$, we denote its 
	equivalence class in $H^2_{\mathrm{na}}(J,V)$ 
	by $[c]$.
	
	\begin{theorem}\label{thm:ext-cohomology}
		There exists a bijection
		\[
		\Psi \colon 
		H^2_{\mathrm{na}}(J,V) 
		\longrightarrow 
		\mathrm{Ext}(J,V),
		\qquad
		[c] \longmapsto [\mathcal{E}_c],
		\]
		where $\mathcal{E}_c$ denotes the standard split 
		extension associated with the $2$-cocycle 
		$c = (\rho,\theta)$.
	\end{theorem}
	
	\begin{proof}
		We show that $\Psi$ is well defined, injective, 
		and surjective.
		
		\medskip
		\noindent\textit{Well-definedness and injectivity.}
		By Proposition~\ref{prop:equivalent-extensions}, 
		two $2$-cocycles are equivalent if and only if 
		their associated standard split extensions are 
		equivalent. Hence $\Psi$ is well defined, and 
		distinct cohomology classes give rise to 
		non-equivalent extensions, so $\Psi$ is injective.
		
		\medskip
		\noindent\textit{Surjectivity.}
		By Theorem~\ref{thm:every-extension-standard}, 
		every split extension of $J$ by $V$ is equivalent 
		to a standard split extension $\mathcal{E}_c$ 
		for some $2$-cocycle $c \in Z^2_{\mathrm{na}}(J,V)$.
		Hence $\Psi$ is surjective.
		
		\medskip
		Therefore, $\Psi$ is a bijection.
	\end{proof}
	%%%%%%%%%%%%%%%%%%%%%%%%%%%%%%%%%%%%%%%
%%%%%%%%%%%%%%%%%%%%%%%%%%%%%%%%%%%%%%%%%%%%%%%%%%%%%%%	
%%%%%%\section{Classification of Low-Dimensional Hom-Jacobi-Jordan Algebras}
	%%%%%%%%%%%%%%%%%%%%%%%%%%%%%%%%%%%%%%%%%%%%%%%%%%%%%%%%%%%%%%%%%%%%%%%%%%%%%%%%%%%%%
%%%%%%%%%%%%%%%%%%%%%%%%%%%
\section{Application: low-dimensional cohomology and split extensions}\label{sec:application}
%%%%%%%%%%%%%%%%%%%%%%%%%%%%%%%%%%%%%%%%%%%
In this section, we apply the cohomological framework of 
Sections~\ref{sec2} and~\ref{sec3} to compute explicit $2$-cocycles 
of low-dimensional Hom-Jacobi-Jordan algebras, yielding via 
Theorem~\ref{thm:ext-cohomology} all decomposable split extensions, 
and via Proposition~\ref{prop:Jordan} all indecomposable regular 
algebras by twisting classical Jacobi--Jordan algebras. The only 
excluded case is the nonregular one with $\mathrm{Im}(\alpha_M^3) = 0$.
%%%%%%%%%%%%%%%%%%%%%%%%%%%%%%%%%
%%%%%%%%%%%%%%%%%%%%%%%%%%%%%%%%%%%%%%%
%%%%%%%%%%%%%%%%%%%%
	\subsection{Classification and Cohomology of $1$-dimensional 
		Hom-Jacobi-Jordan Algebras}
	\label{subsec:dim1}

We begin with the one-dimensional case, classifying such algebras and computing second cohomology sets with one-dimensional coefficients.

	\begin{proposition}\label{prop:1dim-class}
		Let $(J, [\cdot,\cdot], \alpha)$ be a one-dimensional regular 
		Hom-Jacobi-Jordan algebra with basis $\{e_1\}$. Then $[\cdot,\cdot] = 0$ 
		and there exists $a \in \mathbb{C}^*$ such that $\alpha(e_1) = a\, e_1$. 
		We denote this algebra by $J_1^0(a)$. Moreover, two one-dimensional 
		algebras $J_1^0(a)$ and $J_1^0(a')$ are isomorphic if and only if 
		$a = a'$.
	\end{proposition}
	
	\begin{proof}
		The result follows immediately from \eqref{def:multiplicative} and \eqref{def:jacobi}, since $\dim J = 1$.
	\end{proof}
	%%%%%%%%%%%%%%%%%%%%%%%%%%%%%%%%
	\begin{remark}
		The nonregular one-dimensional case corresponds to $\alpha = 0$,
		and yields exactly two algebras up to isomorphism: the abelian one
		and $J^{1z}_1$ defined by $[e_1, e_1] = e_1$.
	\end{remark}
	%%%%%%%%%%%%%%%%%%%%%%%%%%%%%%%%%%%%%%%%%%%%%%%%%%%%%%%%%%%%%%%%%
	\begin{proposition}\label{prop:1dim-reps-J1}
		Let $\{u_1\}$ be a basis of $J=J_1^0(a)$ and $\{v_1\}$ be a basis of $V=J_1^0(b)$.
		Then every nontrivial $2$-cocycle $(\rho,\theta)$ of $J$ with coefficients in $V$ is equivalent to the cocycle $C^1_{1,1}(a,a^2)$ given by  $\rho(u_1)=0$, and $\theta(u_1,u_1)=v_1$.
	\end{proposition}
	\begin{proof}
		Let $\rho(u_1)v_1 = x_1 v_1$ and $\theta(u_1,u_1) = xv_1$. Since $[u_1,u_1]=0$, \eqref{def:rep2} yields $-2ax_1^2=0$, forcing $x_1=0$ and thus $\rho=0$.  By Remark~\ref{rem:kcochain-eigenvalue}, $\theta(u_1,u_1) \neq 0$ forces $b=a^2$.
 Rescaling $v_1$ sets $x=1$, yielding $C^1_{1,1}(a,a^2)$. Otherwise, the cocycle is trivial.
	\end{proof}	
	%%%%%%%%%%%%%%%%%%%%%%%%%%%%%%%%%%%%%%%%	
	\begin{proposition}\label{prop:1dim-regular-trivial-beta}
		Let $J = J^0_1(a)$ be a regular 1-dimensional Hom-Jacobi-Jordan algebra (i.e., $a \neq 0$) and let $(V, [\cdot, \cdot]_V, \beta)$ be a 1-dimensional module with a trivial twisting map $\beta = 0$. Then any 2-cocycle $(\rho, \theta)$ of $J$ with coefficients in $V$ is trivial.
	\end{proposition}	
	\begin{proof}
		Since $[u_1,u_1]=0$ and $a\ne0$, condition~\eqref{def:rep2} gives 
	$\rho=0$ (as in Proposition~\ref{prop:1dim-reps-J1}); since moreover 
	$\beta=0$, condition~\eqref{def:2cohain} gives $\theta=0$.	
	\end{proof}	
	%%%%%%%%%%%%%%%%%%%%%%%%%%%%%%%%%%%%%%%%%%%%%%%%%
	%
	%%
	%%%%%%%%%%%%%%%%%%%%%%%%%
	%%%%%%%%%%%%%%%%%
	\subsection{Classification and Cohomology of 
		$2$-Dimensional Hom-Jacobi-Jordan Algebras}	\label{subsec:dim2}
		%%%%%%%%%%%%%%%%%%%%%%%%%%%%%%%%%%%%%%%%%%%%%%

We classify $2$-dimensional Hom-Jacobi-Jordan algebras, treating the non-regular non-abelian and regular cases separately, and compute the second cohomology sets $H^2_{\mathrm{na}}(J,V)$ needed for the $3$-dimensional split extensions of Subsection~\ref{subsec:dim3}.

%%%%%%%%%%%%%%%%%%%%%%%%%%%%%%%%%%%%%%%%%%%%%%%%%%%%%%%%%%%%%%%%%%%%%%%%%	
\begin{theorem}\label{thm:2dim-nonregular-unified}
Every $2$-dimensional non-regular non-abelian Hom-Jacobi-Jordan algebra over $\mathbb{C}$ is isomorphic to exactly one of the following pairwise non-isomorphic algebras.

\renewcommand{\arraystretch}{1.3}
\begin{longtable}{|c|l|c|}
	\hline
	\textbf{Name} & \textbf{Nonzero products} & \textbf{Twisting map $\alpha$} \\
	\hline
	\endfirsthead
	\hline
	\textbf{Name} & \textbf{Nonzero products} & \textbf{Twisting map $\alpha$} \\
	\hline
	\endhead
	\hline
	\endfoot
	\hline
	\endlastfoot
	\multicolumn{3}{|c|}{\textit{\small Zero-twist case ($\alpha = 0$)}} \\ 
\hline
$J_2^{1z}(\lambda, \mu)$ & $[e_1, e_1] = \lambda  e_1 + \mu e_2,\ [e_1, e_2] =(1-\lambda) e_2,\ [e_2, e_2] = e_1$ & $0$ \\ \hline
%%%%%%%%%%%%%%%%%%%%%%%%%%%%%%%%%%%%%%%%%%%%%%%%%%%%%%%%%%%%%%
$J_2^{2z}(\lambda)$ & $[e_1, e_1] =\lambda e_2,\ [e_1, e_2] = e_1+ e_2,\ [e_2, e_2] = - e_2$ & $0$ \\ \hline
%%%%%%%%%%%%%%%%%%%%%%%%%%%%%%%%%A_4(c)
$J_2^{3z}(\lambda)$ & $[e_1, e_1] = e_1,\ [e_1, e_2] = \lambda e_2$ & $0$ \\ \hline
$J_2^{4z}$ & $[e_1, e_2] = e_1$ & $0$ \\ \hline
%%%%%%%%%%%%%%%%%%%%%%%%%%%%%%%%%%%%%A_5(c)
$J_2^{5z}$ & $[e_1, e_1] = e_1 + e_2,\ [e_1, e_2] = \tfrac{1}{2} e_2$ & $0$ \\ \hline
%%%%%%%%%%%%%%%%%%%%%%%%%%%%%%%%%%%%%%A_10(c)
$J_2^{6z}$ & $[e_1, e_2] = e_1,\ [e_2, e_2] = -e_2$ & $0$ \\ \hline
%%%%%%%%%%%%%%%%%%%%%%%%%%%%%%%%%%%%%%A_11(c)
$J_2^{7z}$ & $[e_1, e_1] = e_2,\ [e_1, e_2] = e_1,\ [e_2, e_2] = -e_2$ & $0$ \\ \hline
%%%%%%%%%%%%%%%%%%%%%%%%%%%%%%%%%%%%%%%A_12(c)
$J_2^{8z}$ & $[e_1, e_1] = e_2$ & $0$ \\ \hline
	\multicolumn{3}{|c|}{\textit{\small Decomposable case ($\mathrm{rank}(\alpha) = 1$, diagonalizable)}} \\ \hline
	$J_{1,1}^{1n}(a)$ & $[e_2, e_2] = e_2$ & $\mathrm{diag}(a,0),\ a \neq 0$ \\ \hline
	\multicolumn{3}{|c|}{\textit{\small Nilpotent-twist case ($\mathrm{rank}(\alpha) = 1$, nilpotent)}} \\ \hline
	$J_2^{1n}$ & $[e_1, e_1] = e_2$ & $\begin{pmatrix} 0 & 0 \\ 1 & 0 \end{pmatrix}$ \\ \hline
	$J_2^{2n}$ & $[e_1, e_2] = e_2$ & $\begin{pmatrix} 0 & 0 \\ 1 & 0 \end{pmatrix}$ \\ \hline
	\caption{$2$-dimensional non-regular non-abelian Hom-Jacobi-Jordan algebras.}
	\label{tab:hjj-2d-nonregular}
\end{longtable}
\end{theorem}

\begin{proof}
Let $(A, [\cdot, \cdot], \alpha)$ be a $2$-dimensional non-regular HJJ algebra over $\mathbb{C}$. Since $\alpha$ is non-bijective, $\mathrm{rank}(\alpha) \in \{0, 1\}$.

\textbf{Case 1: $\alpha = 0$ (Zero-twist case).}
The homomorphism and Hom-Jacobi conditions hold trivially, reducing the classification to complex $2$-dimensional commutative algebras. Referring to the complete classification of $2$-dimensional complex algebras \cite{Bekbaev:2023}, imposing the commutativity condition $[e_1, e_2] = [e_2, e_1]$ eliminates the non-commutative families $A_1, A_6, A_7, A_8, A_9$. Among the remaining families, commutativity further imposes $\beta_2 = 1-\alpha_1$ on $A_2$ and $A_4$, and $\beta_2 = 1$ on $A_3$, reducing each to a one-parameter family, while $A_5$ collapses to the single point $\alpha_1 = 2/3$; the families $A_{10}, A_{11}, A_{12}$ are commutative for all parameter values. Through explicit identification and, in the cases $A_4$ and $A_5$, suitable linear changes of basis, together with reparametrization when needed, the remaining admissible families are identified with the eight canonical zero-twist Hom-Jacobi-Jordan representatives $J_2^{1z}$--$J_2^{8z}$.

\textbf{Case 2: $\mathrm{rank}(\alpha) = 1$, diagonalizable.}
Up to isomorphism, $\alpha = \mathrm{diag}(a,0)$ with $a \neq 0$. Thus $A = J \oplus V$, where $J \cong J^0_1(a)$ is a $1$-dimensional regular HJJ algebra and $\dim V = 1$ with $\alpha|_V = 0$. By Proposition~\ref{prop:1dim-regular-trivial-beta}, the triviality of $2$-cocycles restricts the non-zero bracket on $V$ to $[e_2, e_2] = e_2$, yielding $J_{1,1}^{1n}(a)$.

\textbf{Case 3: $\mathrm{rank}(\alpha) = 1$, nilpotent.}
In Jordan form, $\alpha(e_1) = e_2$ and $\alpha(e_2) = 0$. Multiplicativity forces $[e_1, e_1] = pe_1 + be_2$, $[e_1, e_2] = de_2$, and $[e_2, e_2] = pe_2$. Evaluating the Hom-Jacobi identity at $(e_1, e_2, e_2)$ forces $p = 0$. Automorphisms commuting with $\alpha$ take the form $\varphi = \begin{pmatrix} x & 0 \\ z & x \end{pmatrix}$ ($x \neq 0$), under which $(b, d)$ normalizes to $(1,0)$ ($J_2^{1n}$) or $(0,1)$ ($J_2^{2n}$), excluding the abelian case $(0,0)$.
\end{proof}

%%%%%%%%%%%%%%%%%%%%%%%%%%%%%%%%%%%%%%%
%%%%%%%%%%%%%%%%%%%%%%%%%%%%%	
%%%%%%%%%%%%%%%%%%%%%%%%%%%%%%%%%%%%%	
	%%%%%%%%%%%%%%%%%%%%%%%%%%%%%%%%%%%%%%
	%\paragraph{Convention on representations.}
	\subheading{Convention on representations.}
	In the classification tables below, when only $\rho(u_1)v_1 = y_1v_2$ (or $\rho(u_1)v_2 = z_1v_1$) is listed, it is understood that the other basis vector is mapped to zero, i.e., $\rho(u_1)v_2 = 0$ (resp. $\rho(u_1)v_1 = 0$).
	\begin{theorem}\label{thm:normal-form-spec}
		Let $\{u_1\}$ be a basis of $J\cong J^0_1(a)$ and $\{v_1, v_2\}$ be a basis of $V$, where
 $[v_1,v_1]_V = \lambda v_2$, $\lambda \in \{0,1\}$, and $\beta = \mathrm{diag}(b,c)$ with $b,c \neq 0$. Write
		\[
		\rho(u_1) = \begin{pmatrix} x_1 & z_1 \\ y_1 & t_1 \end{pmatrix} \quad \text{and} \quad \theta(u_1,u_1) = xv_1 + yv_2.
		\]
		
		Any $2$-cocycle $(\rho,\theta)$ of $J$ with coefficients in $V$ satisfies one of the following forms:

\vspace{0.4cm}
\noindent\textbf{1. Case $\lambda = 0$:}

\begingroup
\footnotesize
\renewcommand{\arraystretch}{1.3}
\setlength{\tabcolsep}{4pt}
\begin{longtable}{
  >{\raggedright\arraybackslash}p{3.4cm}
  >{\raggedright\arraybackslash}p{2.3cm}
  >{\raggedright\arraybackslash}p{4.4cm}
  >{\raggedright\arraybackslash}p{3.2cm}
}

	\toprule
	\textbf{Conditions on $a,b,c$} & \textbf{Sub-conditions} & $\rho(u_1)$ & $\theta(u_1,u_1)$ \\
	\midrule
	\endhead

	\multirow{4}{*}{$a\neq \pm1,\ c\neq ab,\ b\neq ac$}
	& $a^2=b\neq c$ & \multirow{4}{*}{$0$} & $xv_1$ \\
	\cmidrule{2-2}\cmidrule{4-4}
	& $a^2=b=c$ & & $xv_1+yv_2$ \\
	\cmidrule{2-2}\cmidrule{4-4}
	& $a^2\notin\{b,c\}$ & & $0$ \\
	\midrule

	\multirow{3}{*}{$a\neq \pm1,\ c=ab$}
	& $ b=a$ & $\rho(u_1)v_1= y_1v_2$ & $yv_2$ \\
	\cmidrule{2-4}
	& $ b\neq a$ & $\rho(u_1)v_1= y_1v_2$ & $0$ \\
	\cmidrule{2-4}
	& $ b=a^2$ & $\rho(u_1)v_1=0$ & $xv_1$  \\
	%&          &                         &($y_1=0, x\in \C$)\\
	\midrule
	\multirow{3}{*}{$a=1,\ b\neq c$}
	& $b=1$ & \multirow{3}{*}{$0$} & $xv_1$ \\
	\cmidrule{2-2}\cmidrule{4-4}
	& $1\notin\{b,c\}$ & & $0$ \\
	\midrule

	\multirow{5}{*}{$a=1,\ b=c$}
	& \multirow{3}{*}{$b=1$} & $0$ & $xv_1+yv_2$ \\
	\cmidrule{3-4}
	& & $\rho(u_1)v_1= x_1v_1+y_1v_2$,  $\rho(u_1)v_2=-\frac{x_1}{y_1}\left( x_1v_1+y_1v_2\right)$ & $\frac{y}{y_1}\,(x_1v_1+y_1v_2)$ \\
	\cmidrule{3-4}
	& & $\rho(u_1)v_2= z_1v_1$  & $xv_1$ \\
	\cmidrule{2-4}
	& \multirow{2}{*}{$b\neq 1$} & $\rho(u_1)v_1= x_1v_1+y_1v_2$,  $\rho(u_1)v_2=-\frac{x_1}{y_1}\left( x_1v_1+y_1v_2\right)$ & $0$ \\
	\cmidrule{3-4}
	& & $\rho(u_1)v_2= z_1v_1$  & $0$ \\
	\midrule

	\multirow{4}{*}{$a=-1,\ c\neq-b$}
	& $b=1,\ c\neq1$ & \multirow{4}{*}{$0$} & $xv_1$ \\
	\cmidrule{2-2}\cmidrule{4-4}
	& $b=c=1$ & & $xv_1+yv_2$ \\
	\cmidrule{2-2}\cmidrule{4-4}
	& $1\notin\{b,c\}$ & & $0$ \\
	\midrule

	\multirow{7}{*}{$a=-1,\ c=-b$}
	& \multirow{3}{*}{$b=1$} & $0$ & $xv_1$ \\
	\cmidrule{3-4}
	& & $\rho(u_1)v_1= y_1v_2$  & $0$ \\
	\cmidrule{3-4}
	& & $\rho(u_1)v_2= z_1v_1$  & $xv_1$ \\
	\cmidrule{2-4}
	& \multirow{1}{*}{$b\neq \pm 1$} & $\rho(u_1)v_1= y_1v_2$  & $0$ \\
	%\cmidrule
	%& & $\rho(u_1)v_2= z_1v_1$ & $0$ \\
	\bottomrule

	\caption{Form of $2$-cocycles for $\lambda = 0$.}
	\label{tab:normal-forms-lambda-0}\\
\end{longtable}
\endgroup

\vspace{0.4cm}
\noindent\textbf{2. Case $\lambda = 1$:}

\begingroup
\footnotesize
\renewcommand{\arraystretch}{1.3}
\setlength{\tabcolsep}{4pt}
\begin{longtable}{
  >{\raggedright\arraybackslash}p{3.4cm}
  >{\raggedright\arraybackslash}p{2.3cm}
  >{\raggedright\arraybackslash}p{4.4cm}
  >{\raggedright\arraybackslash}p{3.2cm}
}
	\toprule
	\textbf{Conditions on $a,b,c$} & \textbf{Sub-conditions} & $\rho(u_1)$ & $\theta(u_1,u_1)$ \\
	\midrule
	\endhead

	\multirow{2}{*}{$c=ab$} & $b=a$ & $\rho(u_1)v_1=y_1v_2$ & $yv_2$ \\
	\cmidrule{2-4}
	& $b\neq a$ & $\rho(u_1)v_1=y_1v_2$ & $0$ \\
	\midrule
	\multirow{2}{*}{$c\neq ab$} & $c=a^2$ & $0$ & $yv_2$ \\
	\cmidrule{2-4}
	& $c\neq a^2$ & $0$ & $0$ \\
	\bottomrule
	\caption{Form of $2$-cocycles for $\lambda = 1$.}
	\label{tab:normal-forms-lambda-1}
\end{longtable}
\endgroup
\noindent\footnotesize\textit{Any case not listed above follows from a listed one via $(v_1,v_2)\mapsto(v_2,v_1)$.}

\end{theorem}
%	================================================================%	
%	=================================================================
\begin{proof}
	Condition~\eqref{def:rep1} yields $a\rho(u_1)\beta = \beta\rho(u_1)$, which gives
	\begin{equation}\label{eq:proof-cond1}
		x_1 b(a-1) = 0, \quad y_1(ab-c) = 0, \quad z_1(ac-b) = 0, \quad t_1 c(a-1) = 0.
	\end{equation}
	
	Since $[v_1,v_1]_V = \lambda v_2$ with $\lambda \in \{0,1\}$, condition~\eqref{def:rep3} yields
	\begin{equation}\label{eq:proof-cond2}
		a\lambda(z_1 v_1 + t_1 v_2) = -2\lambda b x_1 v_2.
	\end{equation}
	If $\lambda \neq 0$, this forces $z_1 = 0$ and $a t_1 = -2b x_1$.
	
	Condition~\eqref{def:rep2} gives $-2a \rho(u_1)^2 v_1 = b x \lambda v_2$ and $-2a \rho(u_1)^2 v_2 = 0$. Evaluating these equations yields
	\begin{equation}\label{eq:proof-cond3}
		x_1^2 + y_1 z_1 = 0, \quad y_1(x_1+t_1) = -\lambda\frac{bx}{2a}, \quad z_1(x_1+t_1) = 0, \quad t_1^2 + y_1 z_1 = 0.
	\end{equation}
	By Remark~\ref{rem:kcochain-eigenvalue}, $\theta(u_1,u_1) \neq 0$ if and only if $a^2 \in \{b,c\}$, in which case it is an eigenvector of $\beta$ associated to the eigenvalue $a^2$. By condition~\eqref{def:cocycle1}, we have $\rho(u_1)\theta(u_1,u_1) = 0$, that is,
	\begin{equation}\label{eq:proof-cond4}
		x_1 x + z_1 y = 0, \qquad y_1 x + t_1 y = 0.
	\end{equation}
	
	We illustrate the method on one representative case; the remaining entries follow analogously. For $a^2=b\neq c$, $\lambda=0$, $a\neq\pm1$, $c\neq ab$, $b\neq ac$: since $b\neq0$, $c\neq0$, $a\neq1$, $ab\neq c$, and $ac\neq b$, equation~\eqref{eq:proof-cond1} forces $x_1=y_1=z_1=t_1=0$, i.e.\ $\rho(u_1)=0$; equations~\eqref{eq:proof-cond3} and~\eqref{eq:proof-cond4} are then automatically satisfied. By Remark~\ref{rem:kcochain-eigenvalue}, since $a^2=b\neq c$, $\theta(u_1,u_1)$ is an eigenvector of $\beta$ for the eigenvalue $b$, hence $\theta(u_1,u_1)=xv_1$. The remaining cases in Tables~\ref{tab:normal-forms-lambda-0} and~\ref{tab:normal-forms-lambda-1} are treated by the same direct analysis of equations~\eqref{eq:proof-cond1}--\eqref{eq:proof-cond4}.
\end{proof}
%============================================================================
\begin{theorem}\label{thm:normal-form-jordan}
	Let $\{u_1\}$ be a basis of $J\cong J^0_1(a)$ and $\{v_1,v_2\}$ be a basis of $V$, where $[v_1,v_1]_V=\lambda v_2$ with $\lambda\in\{0,1\}$,
 $\displaystyle \beta=\begin{pmatrix}b&0\\ 1&b\end{pmatrix}$, $b\neq0$. Write
	\[
	\rho(u_1)=\begin{pmatrix}x_1&z_1\\ y_1&t_1\end{pmatrix}, \qquad \theta(u_1,u_1)=xv_1+yv_2.
	\]
	Any $2$-cocycle $(\rho,\theta)$ of $J$ with coefficients in $V$ satisfies one of the following forms, \emph{independently of $\lambda$}:
	
	\begin{center}
		\renewcommand{\arraystretch}{1.6}
		\begin{tabular}{cccc}
			\toprule
			\textbf{Conditions on $a,b$} & \textbf{Sub-conditions} & $\rho(u_1)$ & $\theta(u_1,u_1)$ \\
			\midrule
			\multirow{2}{*}{$a\neq1$} & $b=a^2$ & \multirow{2}{*}{$0$} & $yv_2$ \\
			\cmidrule{2-2}\cmidrule{4-4}
			& $b\neq a^2$ & & $0$ \\
			\midrule
			\multirow{2}{*}{$a=1$} & $b=1$ & $	\rho(u_1)v_1=y_1v_2 $ & $yv_2$ \\
			\cmidrule{2-4}
			& $b\neq1$ & $	\rho(u_1)v_1=y_1v_2 $  & $0$\\
			\bottomrule
		\end{tabular}
	%	\caption{Form of $2$-cocycles for $\lambda = 0$.}
		\captionof{table}{Forms of $2$-cocycles of $J^0_1(a)$ with coefficients in $V$, $\beta$ a Jordan block.}
		\label{tab:normal-forms-jordan}
	\end{center}
\end{theorem}
\begin{proof}
	Condition~\eqref{def:rep1} forces $x_1=y_1=z_1=t_1=0$ when $a\neq1$, and 
	$z_1=0$, $t_1=x_1$ when $a=1$. For $a=1$, condition~\eqref{def:rep2} 
	(comparing $v_1$-components) then gives $x_1^2=0$, hence $x_1=t_1=0$. 
	In all cases $\rho(u_1)=\begin{pmatrix}0&0\\ y_1&0\end{pmatrix}$, with 
	$y_1=0$ unless $a=1$. By Remark~\ref{rem:kcochain-eigenvalue}, 
	$\theta(u_1,u_1)=yv_2$ with $y=0$ unless $b=a^2$. 
	Condition~\eqref{def:cocycle1} is automatic. This yields exactly the 
	forms stated in Table~\ref{tab:normal-forms-jordan}.
\end{proof}
	%%%%%%%%%%%%%%%%%%%%%%%%%%%%%%%%%%%%%%	
	%%%%%%%%%%%%%%%%%%%%%%%%%%%%%%%%%%%%%%%%%%%%%%%%%%%%%%%%%%%%%%%%%%%%%%%%%%%%%%%%
	\begin{lemma}\label{prop:cocycle-J1-beta-zero}
Let $J = J_1^0(a)$, with basis $\{u_1\}$, and let $(V, [\cdot,\cdot]_V, \beta)$, where $\beta = \begin{pmatrix} 0 & 0 \\ d & 0 \end{pmatrix}$, be a two-dimensional nilpotent-twist HJJ algebra with basis $\{v_1, v_2\}$. Then every $2$-cocycle $(\rho, \theta)$ of $J$ with values in $V$ satisfies $\theta = 0$, and:
\begin{enumerate}[label=(\roman*)]
    \item If $[V,V]_V \subseteq \langle v_2\rangle$, then any cocycle is  given by
    $\rho(u_1)=\begin{pmatrix}0&0\\ y_1&0\end{pmatrix}$.
    \item If $v_1 \in [V,V]_V$, every $2$-cocycle is trivial ($\rho=0$).
\end{enumerate}
\end{lemma}
	\begin{proof}
			Since $\beta$ is nilpotent and $a \neq 0$, neither $a$ nor $a^2$ is an eigenvalue of $\beta$. By Remark~\ref{rem:kcochain-eigenvalue}, $\theta(u_1,u_1) = 0$. 
		 Condition~\eqref{def:rep2} yields $\rho(u_1)^2=0$, so $\rho(u_1) = \begin{pmatrix} 0 & 0 \\ y_1 & 0 \end{pmatrix}$.
 Condition~\eqref{def:rep3} is automatic if $[V,V]_V \subseteq \langle v_2 \rangle$. If $v_1 \in [V,V]_V$, it forces $y_1=0$, so the cocycle is trivial.
	\end{proof}	
	%%%%%%%%%%%%%%%%%%%%%%%%%%%%%%%%%%%%%%%%%%%
	%%%%%%%%%%%%%%%%%%%%%%%%%%%%%%%%%%%%
	\begin{theorem}\label{thm:cocycles-J1-Jz2-unified}
		Let $\{u_1\}$ be a basis of $J\cong J_1^0(a)$, and let $\{v_1,v_2\}$ be a basis of $V$, where $(V,[\cdot,\cdot]_V,\beta)$ is one of the $2$-dimensional zero-twist Hom-Jacobi-Jordan algebras of Theorem~\ref{thm:2dim-nonregular-unified} and the abelian zero-twist Hom-Jacobi-Jordan algebra denoted $J_2^{0z}$.
		
		For the HJJ algebras $V\cong J_2^{0z}$, $J_2^{8z}$ (i.e. whenever
		$[V,V]_V\subseteq \langle v_2\rangle$), every nontrivial $2$-cocycle $(\rho,\theta)$ of $J$
		with coefficients in $V$ is equivalent to
		\[
		\rho(u_1) = \begin{pmatrix}
			0&0\\1&0
		\end{pmatrix}, \qquad \theta(u_1,u_1) = 0.
		\]
		
		For the other non-regular HJJ algebras (i.e. whenever $v_1\in[V,V]_V$), every $2$-cocycle
		$(\rho,\theta)$ of $J$ with coefficients in $V$ is equivalent to the trivial cocycle, that is,
		$\rho=0$ and $\theta=0$.
	\end{theorem}
\begin{proof}
	This follows directly from Lemma~\ref{prop:cocycle-J1-beta-zero}, by checking in each
	case the position of $[V,V]_V$: it equals $0$ for $J_2^{0z}$ and $\langle v_2\rangle$ for
	$J_2^{8z}$, while it contains $v_1$ for the other cases.
\end{proof}
	%%%%%%%%%%%%%%%%%%%%%%%%%%%%%%%%%%%%%%%%%%%%	
	%As a direct consequence of Theorem~\ref{thm:cocycles-J1-Jz2-unified}, we obtain the following result:
	%%%%%%%%%%%%%%%%%%%%%%%%%%%%%%%%%%%%%%%%%%%%%%
	
	\begin{theorem}\label{thm:cocycle-J1-J2-nilpotent}
		Let $\{u_1\}$ be a basis of $J\cong J_1^0(a)$, $a\neq0$. Let $\{v_1,v_2\}$ be a basis of $V$, where $V$ is one of the $2$-dimensional Hom-Jacobi-Jordan algebras $J_2^{0n}$, $J_2^{1n}$, or $J_2^{2n}$. 
		
		Then every nontrivial $2$-cocycle $(\rho,\theta)$ of $J$ with coefficients in $V$ is equivalent to
		\[
		\rho(u_1) = \begin{pmatrix} 0 & 0 \\ 1 & 0 \end{pmatrix}, \qquad \theta(u_1,u_1)=0.
		\]
	\end{theorem}
	\begin{proof}
For $J_2^{0n}$, $J_2^{1n}$, and $J_2^{2n}$, the twisting map is
$\beta=\begin{pmatrix}0&0\\1&0\end{pmatrix}$ and
$[V,V]_V\subseteq\langle v_2\rangle$ (Table~\ref{tab:hjj-2d-nonregular}).
Lemma~\ref{prop:cocycle-J1-beta-zero}(i) applies with $d=1$,
yielding the unique nontrivial cocycle
$\rho(u_1)=\begin{pmatrix}0&0\\1&0\end{pmatrix}$, $\theta=0$.
\end{proof}
	%%%%%%%%%%%%%%%%%%%%%%%%%%%%%%%%%%%%
	%%%%%%%%%%%%%%%%%%%%%%%%%%%%%%%%%%%%%%%%%%%%%%%%%%%%%%%%%%%%%%%%%%
	We now turn to the regular case.
	\begin{theorem}\label{thm:2dim-regular}
		Every $2$-dimensional regular Hom-Jacobi-Jordan algebra is isomorphic to exactly one of the 
		following pairwise non-isomorphic algebras:
	\begin{table}[H] \centering \renewcommand{\arraystretch}{1.6} \begin{tabular}{|c|c|l|} \hline	
		
				\textbf{Name} & \textbf{Twisting map} & \textbf{Nonzero products} \\
				\hline
				$J_{1,1}^{0}(a,b)$ & $\alpha(e_1) = a e_1,\ \alpha(e_2) = b e_2$ & None (Abelian) \\
				\hline
				$J_{1,1}^{1}(a)=J_{1}^{0}(a)\oplus J_{1}^{0}(a^2)$ & $\alpha(e_1) = a e_1,\ \alpha(e_2) = a^2 e_2$ & $[e_1,e_1] = e_2$ \\
				\hline
				$J_{2}^{0}(a)$   & $\alpha(e_1) = ae_1+ e_2,\ \alpha(e_2) = ae_2 $ & None (Abelian) \\
				\hline
				$J_{2}^{1}$   & $\alpha(e_1) = e_1+e_2,\ \alpha(e_2) =  e_2$ & $[e_1,e_1] = e_2$ \\
				\hline

\end{tabular} \caption{2-dimensional regular Hom-Jacobi-Jordan algebras $(a,b\in\mathbb{C}^{\times})$.} \end{table}			
	\end{theorem}
	
	\begin{proof}
		Let $(M, [\cdot,\cdot]_M, \alpha_M)$ be a $2$-dimensional regular Hom-Jacobi-Jordan algebra. 
		Since $\alpha_M$ is invertible, we consider its two possible Jordan canonical forms over $\mathbb{C}$:
		\[
		\begin{pmatrix} a & 0 \\ 0 & b \end{pmatrix}
		\quad \text{and} \quad
		\begin{pmatrix} a & 0 \\ 1 & a \end{pmatrix},
		\qquad \text{with } a, b \in \mathbb{C}^\times.
		\]
		
		\medskip
		\noindent\textbf{Case 1:} $\alpha_M$ is diagonalizable.
		
%		Let $J = \langle e_1\rangle$ and $V = \langle e_2\rangle$ be $\alpha_M$-invariant subspaces,
		with $\alpha(e_1) = a e_1$ and $\beta(e_2) = b e_2$.

	 	By the decomposition $M=J\oplus V$ and
	 Proposition~\ref{prop:1dim-reps-J1}, the bracket relations reduce to $[e_1, e_1]_M = e_2$ with
		$b = a^2$, and $[e_1, e_2]_M = 0$. The nontrivial cocycle yields the non-abelian algebra $J_{1,1}^1(a)$, 
		whereas the trivial cocycle yields the abelian algebra $J_{1,1}^0(a,b)$.

		\medskip
		\noindent\textbf{Case 2:} $\alpha_M$ has a single Jordan block.
		
		Here, $\alpha(e_1) = a e_1+e_2$ and $\alpha(e_2) = a e_2$. 
		The case of a trivial bracket yields the abelian algebra $J_2^0(a)$. 
		If the bracket is nontrivial, the multiplicativity condition  forces $a = 1$, since $a \neq 0$ by regularity.
		Since a Hom-Jacobi--Jordan algebra is classical if and only if
		$
		\alpha=\mathrm{Id},
		$
		it follows that $(M,[\cdot,\cdot]_M,\alpha)$ is a classical Jacobi--Jordan algebra if and only if $M=J^0_{1,1}(1,1), $ or $M=J^1_{1,1}(1) $. 
		By Proposition~\ref{prop:Jordan}, the Hom-bracket is derived from the classical one via $[e_1, e_1] = \alpha([e_1,e_1]')$. 
		Since the underlying classical bracket satisfies $[e_1, e_1]' = e_2$, we obtain 
		$
		[e_1, e_1] = \alpha(e_2) = e_2,
		$
		which precisely defines the algebra $J_2^1$.
	\end{proof}
	%%%%%%%%%%%%%%%%%%%%%%%%%%%%%%%%%%%%%%%%%%%%%%%%%%%%%%%%%%%%%%%%%%%%%%%%%%%	
	The following results describe the second cohomology sets associated with the 
	$1$-dimensional Hom-Jacobi-Jordan algebra $J_1^0(a)$, with coefficients in the 
	$2$-dimensional algebras appearing in Theorem~\ref{thm:2dim-regular}. These computations allow us to classify 
	the corresponding $3$-dimensional split extensions.
	
	\begin{proposition}\label{prop:dim-reps-J1}
	Let $\{u_1\}$ be a basis of $J\cong J^0_1(a)$ and $\{v_1, v_2\}$ be a basis of 
	$V \cong J^0_{1,1}(b,c)$, with $\beta = \begin{pmatrix} b & 0 \\ 0 & c \end{pmatrix}$.	
		Then any nontrivial  $2$-cocycle $(\rho,\theta)$ of $J$ with coefficients in $V$ is equivalent to one of the following:
		
		{\renewcommand{\arraystretch}{1.8}
			\begin{table}[H]
				\centering
				\begin{tabular}{|c|c|c|}
					\hline
					\textbf{Name} & $\rho(u_1)$ & $\theta(u_1,u_1)$ \\
					\hline			
					$C^1_{1,2}(a, b, ab)$ & $\rho(u_1)v_1 = v_2$ & $0$ \\    
					\hline	
					$C^2_{1,2}(a, a^2, c)$ & $\rho(u_1) = 0$ & $\theta(u_1,u_1) = v_1$ \\
					\hline	
				\end{tabular}
				\caption{Nontrivial $2$-cocycles of $J_1^0(a)$ with coefficients in $J_{1,1}^0(b,c)$.}
				\label{tab:cocycles-J1-2dim}
			\end{table}
		}
	\end{proposition}	

\begin{proof}
	Here $\lambda=0$, so Table~\ref{tab:normal-forms-lambda-0} applies. We first assume
 $a^2\neq1$, the relevant rows are ``$c=ab$'' and ``$b=ac$'',
	equivalent under $(v_1,v_2)\leftrightarrow(v_2,v_1)$. Taking $c=ab$:
	if $y_1\neq0$, rescaling $v_2$ gives $\rho(u_1)v_1=v_2$, with
	$\theta(u_1,u_1)=0$ unless $b=a$, in which case $h(u_1)=\frac{y}{2y_1}v_1$
	kills $\theta$ by \eqref{eq:equivalent-cocycle2}; this gives
	$C^1_{1,2}(a,b,ab)$. If $y_1=0$ and $b=a^2$, rescaling $v_1$ gives
	$\theta(u_1,u_1)=v_1$, i.e.\ $C^2_{1,2}(a,a^2,c)$. Since $\rho(u_1)=0$
	is coboundary-invariant, the two families are inequivalent. The case
	$a^2=1$ is treated similarly, using the corresponding rows of
	Table~\ref{tab:normal-forms-lambda-0}, and yields the same two families.
\end{proof}
	%%%%%%%%%%%%%%%%%%%%%%%%%%%%%%%%%%%%%%%
	\begin{proposition}\label{prop:dim-reps-J1-on-J2}
	Let $\{u_1\}$ be a basis of $J\cong J_1^0(a)$ and $\{v_1,v_2\}$ be a basis of 
	$V \cong J_{1,1}^1(b)$. 	
		Then any nontrivial $2$-cocycle $(\rho,\theta)$ of $J$ with coefficients in $V$ is equivalent to one of the following:
		\begin{table}[ht] \centering \renewcommand{\arraystretch}{1.6} \begin{tabular}{|c|c|c|} \hline
	
				\textbf{Name} & $\rho(u_1)$ & $\theta(u_1,u_1)$ \\
				\hline
				$C^3_{1,2}(a,b,a^2)$, $b \in \{-a,a\}$ 
				& $0$ 
				& $v_2$ \\
				\hline
		\end{tabular} \caption{$2$-cocycles of $J_1^0(a)$ with coefficients in $J_{1,1}^1(b)$.} \label{tab:cocycles-J1-on-J2} \end{table}	
	\end{proposition}
	\begin{proof}
Since $V\cong J_{1,1}^1(b)$, we have $c=b^2$.

\textbf{If $c=ab$:} then $c=b^2$ forces $a=b$. By Table~\ref{tab:normal-forms-lambda-1}, $\rho(u_1)v_1=y_1v_2$, $\theta(u_1,u_1)=yv_2$. The coboundary $h(u_1)=y_1v_1$ cancels $\rho$, leaving $\rho'=0$, $\theta(u_1,u_1)=yv_2$.

\textbf{If $c\neq ab$:} then $b\neq a$, and $\rho(u_1)=0$; also $\theta(u_1,u_1)=yv_2$ iff $c=a^2$, i.e.\ $b^2=a^2$, i.e.\ $b=-a$ (since $b\neq a$); otherwise $\theta=0$.

In both cases the only nontrivial class arises for $b\in\{-a,a\}$, giving $\rho=0$, $\theta(u_1,u_1)=v_2$ after rescaling — i.e.\ $C^3_{1,2}(a,b,a^2)$, $b\in\{-a,a\}$.
\end{proof}
%%%%%%%%%%%%%%%%%%%%%%%%%%
	%%%%%%%%%%%%%%%%%%%%%%%%%%%%%%%%%%%%%%%%%%%%%%%%%%%%%%%%%%%%%%%%%%%%		
\begin{proposition}\label{prop:dim-reps-J1-on-J3}
	Let $\{u_1\}$ be a basis of $J\cong J^0_1(a)$ and $\{v_1, v_2\}$ 
	be a basis of $V $.
	Then any  nontrivial $2$-cocycle $(\rho,\theta)$ of $J$ 
	with coefficients in $V$ is equivalent to one 
	of the following:
	{\renewcommand{\arraystretch}{1.8}
		\begin{table}[H]
			\centering
			\begin{tabular}{|c|c|c|}
				\hline
				\textbf{Name} & $\rho(u_1)$ & $\theta(u_1,u_1)$ \\
				\hline
				$C^1_{1,2}(a, a^2)$, 
				& $\rho(u_1) = 0$
				& $\theta(u_1,u_1) = v_2,$ \\
				\hline
				
				$C^2_{1,2}(1, b)$, 
				& $\rho(u_1)v_1 = v_2,$
				& $\theta(u_1,u_1) = 0$ \\
				\hline
				$C^3_{1,2}(1, 1)$, 
				& $\rho(u_1)v_1 = v_2,$
				& $\theta(u_1,u_1) = v_2$ \\
				\hline
				
			\end{tabular}
			\caption{Nontrivial $2$-cocycles of $J_1^0(a)$ with coefficients in $J_2^0(b)$ or $J_2^1$ ($\dim V=2$).}
\label{tab:cocycles-J1-on-J3}
		\end{table}
	}
\end{proposition}
\begin{proof}
Specializing Table~\ref{tab:normal-forms-jordan} at $\lambda\in\{0,1\}$: if $a\neq1,\,b=a^2$, we get $C^1_{1,2}(a,a^2)$ ($\rho=0,\,\theta=v_2$); if $a=1,\,b\neq1$, rescaling $y_1$ gives $C^2_{1,2}(1,b)$ ($\rho(u_1)v_1=v_2,\,\theta=0$); if $a=1,\,b=1$, rescaling $(y_1,y)\in\{(1,0),(0,1),(1,1)\}$ gives $C^2_{1,2}(1,1)$, $C^1_{1,2}(1,1)$, $C^3_{1,2}(1,1)$ respectively. All other entries vanish (trivial cocycle).
\end{proof}
%%%%%%%%%%%%%%%%%%%%%%%
\begin{proposition}\label{prop:dim-reps-J2-on-J1}
	Let $\{u_1,u_2\}$ be a basis of $J$, where $J$ is one of the
	two-dimensional regular Hom-Jacobi-Jordan algebras $J^0_2(b)$ of Theorem~\ref{thm:2dim-regular}, and let $\{v_1\}$ be a
	basis of $V\cong J^0_1(a)$.
	Then any nontrivial $2$-cocycle $(\rho,\theta)$ of $J$
	with coefficients in $V$ is equivalent to the following:
	{\renewcommand{\arraystretch}{1.8}
		\begin{table}[H]
			\centering
			\begin{tabular}{|c|c|c|c|}
				\hline
				\textbf{Name} & \textbf{Condition} & $\rho(u_1)$ & $\theta(u_1,u_1)$ \\
				\hline
				$C^1_{2,1}(b,b^2)$,
				& $a=b^2$
				& $\rho=0$
				& $\theta(u_1,u_1) = v_1$ \\
				\hline
			\end{tabular}
			\caption{Nontrivial $2$-cocycle of $J^0_2(b)$  with coefficients in $J^0_1(a)$ ($\dim V=1$).}
			\label{tab:cocycles-J2-on-J1}
		\end{table}
	}
\end{proposition}
\begin{proof}
	Write $\rho(u_i)v_1=x_iv_1$, $\theta(u_i,u_j)=y_{ij}v_1$. Condition~\eqref{def:rep2} forces $x_1=x_2=0$ (using $[u_1,u_1]=0$), so $\rho=0$. Condition~\eqref{def:2cohain} then reduces to the triangular system $(b^2-a)y_{22}=0$, $(b^2-a)y_{12}+by_{22}=0$, $(b^2-a)y_{11}+2by_{12}+y_{22}=0$, whose solution vanishes unless $a=b^2$, in which case $y_{12}=y_{22}=0$ and, after rescaling $y_{11}$, we obtain $C^1_{2,1}(b,b^2)$. Condition~\eqref{def:cocycle1} is then automatic, since it only involves $\theta(\cdot,u_2)=y_{12}v_1=0$ and $\theta(u_2,u_2)=y_{22}v_1=0$.
\end{proof}

	%%%%%%%%%%%%%%%%%%%%%%%%

	%%%%%%%%%%%%%%%%%%%%%%%%%%%%%%%%%%%%%%%%%%%%%%%%%%%%%%%

	%%%%%%%%%%%%%%%%%%%%%%%%%%%%%
	We now turn to the classification of $2$-cocycles for the case where $J$ is a 
	$2$-dimensional \textbf{regular} Hom-Jacobi-Jordan algebra and $V$ is a $1$-dimensional 
	module with a trivial twisting map.
	
	\begin{proposition}\label{prop:2dim-regular-reps-1dim}
		Let $J$ be a two-dimensional regular Hom-Jacobi-Jordan 
		algebra and let $(V, [\cdot,\cdot]_V, 0)$ be a one-dimensional module with a trivial twisting map. 
		Then any $2$-cocycle $(\rho, \theta)$ of $J$ 
		with coefficients in $V$ is equivalent to the 
		trivial cocycle, that is,
		$
		\rho = 0 \quad \text{and} \quad \theta = 0.
		$
	\end{proposition}
	\begin{proof}
		Let $\{u_1, u_2\}$ and $\{v_1\}$ be bases for $J$ and $V$, respectively, and write $\rho(u_i) = x_i v_1$. Up to a change of basis, the automorphism $\alpha$ takes either a diagonal form $\mathrm{diag}(a,b)$ with $a,b \neq 0$, or a non-diagonal Jordan canonical form $\left(\begin{smallmatrix} a & 0 \\ 1 & a \end{smallmatrix}\right)$ with $a \neq 0$.
		
		In the diagonal case, condition~\eqref{def:2cohain} with $\beta = 0$ yields
		\[
		a^2\theta(u_1,u_1) = b^2\theta(u_2,u_2) = ab\theta(u_1,u_2) = ab\theta(u_2,u_1) = 0,
		\]
		which directly gives $\theta = 0$ since $a,b \neq 0$. Substituting $\theta = 0$ into~\eqref{def:rep2} reduces to $ax_1^2 = bx_2^2 = 0$, forcing $x_1 = x_2 = 0$, so $\rho = 0$.
		
		Similarly, when $\alpha$ is a Jordan block, evaluating~\eqref{def:2cohain} and~\eqref{def:rep2} leads to an analogous triangular system whose sequential resolution yields $\theta = 0$ and $\rho = 0$.
	\end{proof}
	%\begin{rem}
	%	This result shows that when $\beta = 0$, 
	%	there are no nontrivial extensions of a 
	%	\textbf{regular} two-dimensional Hom-Jacobi-Jordan algebra 
	%	by a one-dimensional module. 
	%\end{rem}	
	%%%%%%%%%%%%%%%%%%%%%%%%%%%%%%%%%%%%%%%%%%%%
	%%%%%%%%%%%%%%%%%%%%%%%%%%%%%%%%%
	%%%%%%%%%%%%%%%%%%%%%%%%%%%%%%%%%
	\subsection{$3$-dimensional split extensions}
	\label{subsec:dim3}
	%%%%%%%%%%%%%%%%%%%%%%%
We now apply the cohomological framework of the preceding sections to the $3$-dimensional case, evaluating the $2$-cocycles computed in Subsections~\ref{subsec:dim1} and~\ref{subsec:dim2}. Pairwise non-isomorphism in the resulting classification tables is guaranteed by the bijection between $H^2_{\mathrm{na}}(J,V)$ and $\mathrm{Ext}(J,V)$ (Theorem~\ref{thm:ext-cohomology}), the isomorphism invariance of the Fitting decomposition (Lemma~\ref{lem:fitting-invariant}), and standard algebraic invariants such as $\dim M^2$ and continuous parameters.

\begin{lemma}\label{lem:fitting-invariant}
	Let $(M,[\cdot,\cdot]_M,\alpha_M)$ and $(M',[\cdot,\cdot]_{M'},\alpha_{M'})$ be isomorphic finite-dimensional Hom-Jacobi-Jordan algebras, via $\varphi:M\to M'$. Then $\varphi(\operatorname{Im}(\alpha_M^k))=\operatorname{Im}(\alpha_{M'}^k)$ and $\varphi(\ker(\alpha_M^k))=\ker(\alpha_{M'}^k)$ for all $k\ge 0$. In particular, $\dim\operatorname{Im}(\alpha_M^k)$ is an isomorphism invariant.
\end{lemma}
\begin{proof}
	Since $\varphi\circ\alpha_M=\alpha_{M'}\circ\varphi$ and $\varphi$ is a linear bijection, $\varphi\circ\alpha_M^k=\alpha_{M'}^k\circ\varphi$ for all $k$, which gives both identities directly.
\end{proof}

\begin{theorem}\label{thm:main-classification}
	Every non-abelian $3$-dimensional non-regular Hom-Jacobi-Jordan algebra
	over $\mathbb{C}$ whose regular part is nonzero, i.e.\
	$\operatorname{Im}(\alpha_M^3)\neq 0$, is isomorphic to exactly one
	of the following pairwise non-isomorphic algebras.
	\begin{longtable}{|c|p{10cm}|}
	\hline
	
	\multicolumn{2}{|c|}{\small\itshape Decomposable case ($1$D regular $\oplus$ $2$D zero-twist) with $\alpha = \mathrm{diag}(a,0,0)$} \\
	\hline
	$J^{n1}_{1,2}=J^0_1(a)\oplus J^{0z}_2$ & $[e_1,e_2]=e_3$ \\ \hline
	$J^{n2}_{1,2}=J^0_1(a)\oplus J^{1z}_2(\lambda, \mu)$ & $ [e_2, e_2] = \lambda  e_2 + \mu e_3,\ [e_2, e_3] =(1-\lambda) e_3,\ [e_3, e_3] = e_2$ \\ \hline
	$J^{n3}_{1,2}=J^0_1(a)\oplus J^{2z}_2(\lambda)$ & $ [e_2,e_2]=\lambda e_3,\ [e_2,e_3]=e_2+e_3,\ [e_3,e_3]=-e_3$ \\ \hline
	%$J^{n4}_{1,2}=J^0_1(a)\oplus J^{3z}_2(0)$ & $[e_1,e_2]=z\,e_3,\ z\in\{0,1\},\ [e_2,e_2]=e_2$ \\ \hline
	$J^{n4}_{1,2}(\mu)=J^0_1(a)\oplus J^{3z}_2(\mu)$ & $[e_2,e_2]=e_2,\ [e_2,e_3]=\mu e_3$ \\ \hline
	%%%%%%%%%%%%%%%%%%%%%%%%%%%%%%%%%%%%%%%%%%%%%
	$J^{n5}_{1,2}=J^0_1(a)\oplus J^{4z}_2$ & $ [e_2,e_3]= e_2,$ \\ \hline
	%%%%%%%%%%%%%%%%%%%%%%%%%%%%%%%%%%%%%%%%%%%%%%%%
	$J^{n6}_{1,2}=J^0_1(a)\oplus J^{5z}_2$ & $[e_2,e_2]=e_2+e_3,\ [e_2,e_3]=\tfrac12 e_3$ 
	\\ \hline
	%%%%%%%%%%%%%%%%%%%%%%%%%%%%%%%%%%%%%%%%%%%%%%%%%%%%%%%%%%%
	$J^{n7}_{1,2}=J^0_1(a)\oplus J^{6z}_2$ & $[e_2,e_3]=e_2,\ [e_3,e_3]=-e_3$ \\ \hline	
	%%%%%%%%%%%%%%%%%%%%%%%%%%%%%%%%%%%%%%%%%%%%%%%%%%%%%%%%%%%
	$J^{n8}_{1,2}=J^0_1(a)\oplus J^{7z}_2$ & $[e_2,e_2]=e_3,\ [e_2,e_3]=e_2,\ [e_3,e_3]=-e_3$ \\ \hline	
	%%%%%%%%%%%%%%%%%%%%%%%%%%%%%%%%%%%%%%%%%%%%%%%%%%%%%%%%%%%%
	$J^{n9}_{1,2}=J^0_1(a)\oplus J^{8z}_2$ & $[e_1,e_2]=te_3,\quad t\in\{0,1\} ,\ [e_2,e_2]=e_3$ \\ \hline	
	\multicolumn{2}{|c|}{\small\itshape Decomposable case ($1$D regular $\oplus$ $2$D nilpotent) with $\alpha = \left(\begin{smallmatrix}a&0&0\\0&0&0\\0&1&0\end{smallmatrix}\right)$} \\
	\hline
	$J^{m1}_{1,2}=J^0_1(a)\oplus J_2^{0n}$ & $[e_1,e_2]=e_3$ \\ \hline
	$J^{m2}_{1,2}=J^0_1(a)\oplus J_2^{1n}$ & $[e_1,e_2]=te_3,\ t\in\{0,1\},\ [e_2,e_2]=e_3$ \\ \hline
	%	$J^{m3}_{1,2}=J^0_1(a)\oplus J_2^{1n}$ & $[e_1,e_2]=e_3,\ [e_2,e_2]=e_3$ \\ \hline
	%$J^{m4}_{1,2}=J^0_1(a)\oplus J_2^{2n}$ & $[e_2,e_3]=e_3$ \\ \hline
	$J^{m3}_{1,2}=J^0_1(a)\oplus J_2^{2n}$ & $[e_1,e_2]=te_3,\ t\in\{0,1\},\,\ [e_2,e_3]=e_3$ \\ \hline \hline
	
	\multicolumn{2}{|c|}{\small\itshape Decomposable case ($2$D regular $\oplus$ $1$D zero-twist)} \\
	\hline
	\multicolumn{2}{|c|}{\small\itshape Case 1: $\alpha =\mathrm{diag}(a,a^2,0)\ (a\neq0)$} \\
	\hline
	$J^{n1}_{2,1}=J^1_{1,1}(a)\oplus J_1^{0z}$ & $[e_1,e_1]=e_2,\ [e_3,e_3]=te_3,\ \ t\in\{0,1\}$ \\ \hline
	\multicolumn{2}{|c|}{\small\itshape Case 2: $\alpha = \left(\begin{smallmatrix}1&0&0\\1&1&0\\0&0&0\end{smallmatrix}\right)$} \\
	\hline
	$J^{n2}_{2,1}=J^1_{2}\oplus J_1^{0z}$ & $[e_1,e_1]=e_2,\ [e_3,e_3]=te_3,\ \ t\in\{0,1\}$ \\ \hline
	\caption{$3$-dimensional non-regular non-abelian Hom-Jacobi-Jordan algebras over $\mathbb{C}$ with nonzero regular part.}\label{tab:hjj-3d-nonregular}
\end{longtable}
\end{theorem}
	\begin{proof}
		Let $(M,[\cdot,\cdot]_M,\alpha_M)$ be a non-abelian $3$-dimensional non-regular HJJ algebra over $\mathbb{C}$ with $\operatorname{Im}(\alpha_M^3)\neq 0$. By the Fitting decomposition (Theorem~\ref{thm:decomposition_HJJ}), $M=J\oplus K$ where $J=\operatorname{Im}(\alpha_M^3)$ is regular and $K=\operatorname{Ker}(\alpha_M^3)$ is an ideal. Since $\operatorname{Im}(\alpha_M^3)\neq 0$ and $M$ is non-regular, $\dim J\in\{1,2\}$.
		
		\textbf{Case 1: $\dim J=1$.} Then $J\cong J_1^0(a)$ for some $a\in\mathbb C^*$, and $K$ is a $2$-dimensional ideal. The restriction $\beta=\alpha_M|_K$ is non-invertible. Hence, by the $2$-dimensional classification, $\beta$ is either zero (zero-twist case) or nilpotent nonzero (nilpotent case).

	%	=================================

	\textit{Subcase 1.1: $1$D regular $\oplus$ $2$D zero-twist.} Here $\beta=0$. By Theorem~\ref{thm:cocycles-J1-Jz2-unified}, the $2$-cocycles are either trivial or equivalent to the nontrivial cocycle $\rho(u_1)v_1=v_2$, $\theta=0$, according to the position of $[K,K]$. Combining these cocycles with the $2$-dimensional zero-twist algebras yields the families $J^{n1}_{1,2}$, $J^{n2}_{1,2}$, $J^{n3}_{1,2}$, $J^{n4}_{1,2}(\mu)$, $J^{n5}_{1,2}$, $J^{n6}_{1,2}$, $J^{n7}_{1,2}$, $J^{n8}_{1,2}$, and $J^{n9}_{1,2}$ displayed in the table.
	
	\textit{Subcase 1.2: $1$D regular $\oplus$ $2$D nilpotent.} Here $\beta$ is nilpotent nonzero. By Theorem~\ref{thm:cocycle-J1-J2-nilpotent}, the $2$-cocycles are either trivial or equivalent to the nontrivial cocycle $\rho(u_1)v_1=v_2$, $\rho(u_1)v_2=0$, $\theta=0$. Combining these cocycles with $K\cong J_2^{0n},J_2^{1n},J_2^{2n}$ yields the algebras $J^{m1}_{1,2}$, $J^{m2}_{1,2}$, and $J^{m3}_{1,2}$.
	%========================
	
	\textbf{Case 2: $\dim J=2$.} Then $K$ is $1$-dimensional and $\beta=\alpha_M|_K=0$. By Proposition~\ref{prop:2dim-regular-reps-1dim}, every $2$-cocycle is trivial, so $M\cong J\oplus K$. The non-abelian $2$-dimensional regular algebras $J_{1,1}^1$ and $J_2^1$ yield $J^{n1}_{2,1}$ and $J^{n2}_{2,1}$.

	%=================
%	\textbf{Non-isomorphism.} The cases are distinguished by $\alpha_M$: in Case~1, Subcase~1.1 has $\beta=0$, while Subcase~1.2 has $\beta$ nilpotent nonzero; Case~2 is characterized by $\dim(\operatorname{Im}\alpha_M^3)=2$. Within each subcase, bracket invariants ($\dim M^2$, parameters $\mu,\lambda$) ensure pairwise non-isomorphism.
	
	%========================================	
	\textbf{Non-isomorphism.} By Lemma~\ref{lem:fitting-invariant}, the cases are distinguished by $\alpha_M$: in Case~1, Subcase~1.1 has $\beta=0$, while Subcase~1.2 has $\beta$ nilpotent nonzero; Case~2 is characterized by $\dim(\operatorname{Im}\alpha_M^3)=2$. Within each subcase, bracket invariants ($\dim M^2$, parameters $\mu,\lambda$) ensure pairwise non-isomorphism.

	\end{proof}
	%%%%%%%%%%%%%%%%%%%%%%%%%%%%%%%%%%%%%%%%%%%%%%%%%%%%%%%%%%%%%
	
	%=================================
	
	We now present the classification of non-abelian regular
	$3$-dimensional Hom-Jacobi-Jordan algebras.
%	============================

	\begin{proposition}\label{prop:3dim-regular}
		Every non-abelian $3$-dimensional regular Hom-Jacobi-Jordan 
		algebra is isomorphic to exactly one of the 
		following pairwise non-isomorphic algebras.	
		%============================
		The decomposable algebras are denoted by $J^k_{i,j}$, where
		$i = \dim J$, $j = \dim V$, and $k$ is the index; the indecomposable algebras are denoted by $J^k_3$.	Unless otherwise specified, $a,b,c\in\mathbb{C}^*$.	
	%	================	
		\begin{enumerate}[label=$\bullet$]

			\item $J^1_{1,2}$: $J^0_1(a) \oplus J^0_{1,1}(b,ab)$,\quad
			$[e_1,e_2] = e_3$,\quad
			$\alpha = \mathrm{diag}\!\left(a,b,ab\right)$.
			%%%%%%%%%%%%%%%%%%%%%%%
			
			\item $J^2_{1,2}$: $J^0_1(a) \oplus J^0_{1,1}(a^2,c)$,\quad
			$[e_1,e_1] = e_2$,\quad
			$\alpha = \mathrm{diag}(a,a^2,c)$.
						%\item $J^3_{1,2}$: $J^0_1(a) \oplus J^1_{1,1}(-a)$,\quad
		 %$[e_2,e_2] = e_3$,\quad
			%$\alpha = \mathrm{diag}(a,-a,a^2)$.
		%%%%%%%%%%%%%%%%%%%%%%%%%%%%%%%%%%%%%	
			\item $J^3_{1,2}$: $J^0_1(a) \oplus J^1_{1,1}(-a)$,\quad
			$[e_1,e_1] = e_3$,\ $[e_2,e_2] = e_3$,\quad
			$\alpha = \mathrm{diag}(a,-a,a^2)$.
			%%%%%%%%%%%%%%%%%%%%%%%%%%%%%%%%%%%%%%%%%%%

			\item $J^4_{1,2}$: $J^0_1(a) \oplus J^0_2(a^2)$,\quad
			$[e_1,e_1] = e_3$,\quad
			$\alpha = \begin{pmatrix} a&0&0\\0&a^2&0\\
				0&1&a^2 \end{pmatrix}$.

			\item $J^{5}_{1,2}$: $J^0_1(1) \oplus J^0_2(b)$,\quad
			$[e_1,e_2] = e_3$,\quad
			$\alpha = \begin{pmatrix} 1&0&0\\0&b&0\\
				0&1&b \end{pmatrix}$.
			\item $J^{6}_{1,2}$: $J^0_1(1) \oplus J^0_2(1)$,\quad
			$[e_1,e_1] = e_3$,\quad $[e_1,e_2] = e_3$,\quad
			$\alpha = \begin{pmatrix} 1&0&0\\0&1&0\\
				0&1&1 \end{pmatrix}$.
			%%%%%%%%%%%%%%%%%%%%%%%%%%%%%%%%%%%%%%%%%%%%%%%%%88888888888888888888888888888888888888888888888888888888
			\item $J^{7}_{1,2}$: $J^0_1(a)\oplus J^1_2$,\quad$[e_2,e_2] = e_3$,\quad
			$\alpha = \begin{pmatrix} a&0&0\\0&1&0\\
				0&1&1 \end{pmatrix}$ .
			%%%%%%%%%%%%%%%%%%%%%%%%%%%%%%%%%%%%%%%%%%%%%%%%%	9999999999999999999999999999999999999999
		%	\item $J^{9}_{1,2}$: $J^0_1(1) \oplus J^1_2$,\quad$[e_1,e_2] = e_3$,\quad
			%$[e_2,e_2] = e_3$,\quad
			%$\alpha = \begin{pmatrix} 1&0&0\\0&1&0\\
				%0&1&1 \end{pmatrix}$.
			%%%%%%%%%%%%%%%%%%%%%%%%%%%%%%%%%%%%%%%%%%%%%%%%%	10101010101010
			%\item $J^{10}_{1,2}$: $J^0_1(1) \oplus J^1_2$, $\quad[e_1,e_1] = e_3$, $\quad[e_1,e_2] = e_3$,\quad
			%$[e_2,e_2] = e_3$,\quad
			%$\alpha = \begin{pmatrix} 1&0&0\\0&1&0\\
				%0&1&1 \end{pmatrix}$.
			%%%%%%%%%%%%%%%%%%%%%%%%%%%%%%%%%%%%%%%%%%%%%%%%%	11111111
			\item $J^{8}_{1,2}$: $J^0_1(-1) \oplus J^1_2$,\quad$[e_1,e_1] = e_3$,\quad
			$[e_2,e_2] = e_3$,\quad
			$\alpha = \begin{pmatrix} -1&0&0\\0&1&0\\
				0&1&1 \end{pmatrix}$.
\item $J^{9}_{1,2} = J^0_1(1)\oplus J^1_2$: 
$[e_1,e_1] = e_3$,\ $[e_1,e_2] = e_3$,\ $[e_2,e_2] = e_3$,\quad
$\alpha = \begin{pmatrix} 1&0&0\\0&1&0\\
	0&1&1 \end{pmatrix}$.

			%%%%%%%%%%%%%%%%%%%%%%%%%%%%%%%%%%%%%%%%%%%%%%%%%		
			\item $J^{1}_{2,1}$: $J^0_2(b) \oplus J^0_1(b^2)$,\quad$[e_2,e_2] = e_1$,\quad
		$\alpha = \begin{pmatrix} b^2&0&0\\0&b&0\\
			0&1&b \end{pmatrix}$.
			\item $J^{1}_{3}$: $[e_1,e_1] = e_3$,\quad
			$\alpha = \begin{pmatrix} 1&0&0\\1&1&0\\
				0&1&1 \end{pmatrix}$.			
			%\item $J^{2}_{3}$: $[e_1,e_1] =e_2+ e_3$,\ 
		%	\quad
		%	$\alpha = \begin{pmatrix} 1&0&0\\1&1&0\\
			%	0&1&1 \end{pmatrix}$.
			
		\end{enumerate}
	\end{proposition}
\begin{proof}
	  Let $(M,[\cdot,\cdot]_M,\alpha_M)$ be a non-abelian $3$-dimensional regular HJJ algebra. Write $\alpha=\alpha_M$. Since $\alpha$ is invertible, after a change of basis its Jordan canonical form over $\mathbb{C}$ is one of
	  \[
	  \mathrm{diag}(a,b,c),\qquad \begin{pmatrix} a&0&0\\0&b&0\\0&1&b\end{pmatrix},\qquad \begin{pmatrix} a&0&0\\1&a&0\\0&1&a\end{pmatrix},
	  \]
	  with $a,b,c\in\mathbb{C}^\times$.  By Theorem~\ref{thm:decomposition-regular}, there exists a nontrivial proper ideal $V$ of $M$ such that $D(M) \subseteq V$.

	\noindent\textbf{Case 1:} $\alpha = \mathrm{diag}(a,b,c)$.
	
Since $\alpha$ is diagonal, every subspace is $\alpha$-invariant, and any subspace containing $D(M)$ is automatically an ideal. By Theorem~\ref{thm:decomposition-regular}, we may thus write $M=J\oplus V$ with $J$ abelian, $D(M)\subseteq V$, and $\dim J=1$, $\dim V=2$; hence $J\cong J^0_1(a)$ by Proposition~\ref{prop:1dim-class}. Moreover, by
Proposition~\ref{prop:dim-reps-J1}, the non-abelian extensions with diagonal
$\alpha$ are, up to isomorphism, of type $J^k_{1,2}$ for $k=1,2$.

When $b=-a$, the cocycle $C^3_{1,2}(a,b,a^2)$ of
Proposition~\ref{prop:dim-reps-J1-on-J2} yields the family $J^3_{1,2}$.
%%%%%%%%%%%%%%%%%%%%%%%%%

When $b=a$, write $c=a^2$. For the cocycle $C^3_{1,2}(a,a,a^2)$ we have $\rho=0$, $\theta(u_1,u_1)=v_2$, hence $[u_1,u_1]=[v_1,v_1]=v_2$. Setting
\[
e_1=u_1+iv_1,\qquad e_2=u_1-iv_1,\qquad e_3=2v_2,
\]
a direct computation gives $[e_1,e_2]=e_3$, all other products vanishing, while $\alpha(e_1)=ae_1$, $\alpha(e_2)=ae_2$, $\alpha(e_3)=a^2e_3$. Hence $M\cong J^1_{1,2}$; so $b=a$ yields no new isomorphism class.

%%%%%%%%%%%%%%%%%%%%%%%

	  \noindent\textbf{Case 2:} $\alpha = \begin{pmatrix} a&0&0\\0&b&0\\0&1&b\end{pmatrix}$. The $\alpha$-invariant subspaces are $\langle e_1\rangle$, $\langle e_3\rangle$, $\langle e_1,e_3\rangle$, and $\langle e_2,e_3\rangle$; we consider in turn $V=\langle e_2,e_3\rangle$, $V=\langle e_1,e_3\rangle$, $V=\langle e_1\rangle$, and $V=\langle e_3\rangle$.
	  
	   \textit{Subcase 2.1: $V = \langle e_2, e_3 \rangle$.} Here $J = \langle e_1 \rangle$ is an $\alpha$-invariant complement of $V$, so $M = J \oplus V$ with $J \cong J^0_1(a)$ by Proposition~\ref{prop:1dim-class}.
	   
	   If $V\cong J^0_2(b)$, Proposition~\ref{prop:dim-reps-J1-on-J3} yields
	   $J^4_{1,2}$ when $b=a^2$, $J^5_{1,2}$ when $a=1$, and $J^6_{1,2}$ when
	   $a=b=1$.
	   
	   If $V\cong J^1_2$, the trivial cocycle gives the direct sum $J^7_{1,2}$,
	   while Proposition~\ref{prop:dim-reps-J1-on-J3} yields $J^8_{1,2}$ when
	   $a=-1$, $J^9_{1,2}$ when $a=b=1$ .

	   \textit{Subcase 2.2: $V=\langle e_1,e_3\rangle$.}
	   If $\dim [M,M]=2$, then $[M,[M,M]]=[M,M]$, contradicting the nilpotency of $M$. Hence $\dim[M,M]=1$. Since $[M,M]\subset V$, we have
	   
	   $$
	   [M,M]=\langle e_1\rangle
	   \quad\text{or}\quad
	   [M,M]=\langle e_3\rangle.
	   $$
	   
	   If $[M,M]=\langle e_3\rangle$, then $V'=\langle e_2,e_3\rangle$ is an $\alpha$-invariant ideal admitting the $\alpha$-invariant complement $\langle e_1\rangle$. Thus this case reduces to Subcase~2.1.

	 If $[M,M]=\langle e_1\rangle$, then $V'=\langle e_1\rangle$ is an $\alpha$-invariant ideal admitting the $\alpha$-invariant complement $\langle e_2,e_3\rangle$; Proposition~\ref{prop:dim-reps-J2-on-J1} yields $J^{1}_{2,1}$.

	\textit{Subcase 2.3: $V=\langle e_1\rangle$.}
Here $J=\langle e_2,e_3\rangle$ is an $\alpha$-invariant complement of $V$, so $M=J\oplus V$ with $J\cong J^0_2(b)$. By Proposition~\ref{prop:dim-reps-J2-on-J1}, we obtain $J^{1}_{2,1}$.   
	   
\textit{Subcase 2.4: $V=\langle e_3\rangle$.}
Here $D(M)=V$. Choosing instead $V'=\langle e_2,e_3\rangle$, an $\alpha$-invariant ideal containing $D(M)$, reduces this case to Subcase~2.1.

	   \textbf{Case 3:} $\alpha$ is a single $3\times3$ Jordan block with 
	   eigenvalue $a\neq0$. By Proposition~\ref{prop:Jordan}, 
	   $[x,y]':=\alpha^{-1}([x,y]_M)$ defines a classical Jacobi-Jordan 
	   algebra, so by Case~1 its structure is $J^1_{1,2}$, $J^2_{1,2}$, or 
	   $J^3_{1,2}$. The algebra $J^3_{1,2}$ is excluded, since 
	   $\mathrm{diag}(a,-a,a^2)$ is not proportional to the identity. 
	   Twisting $J^1_{1,2}$ by this Jordan block gives $[e_1,e_2]=e_3$, but 
	   then $\alpha([e_1,e_1])\neq[\alpha(e_1),\alpha(e_1)]$, so the result 
	   fails multiplicativity and is not an HJJ algebra; the only 
	   $\alpha$-invariant line is $\langle e_3\rangle$, hence $[J,J]=\langle e_3\rangle$. 
	   For $J^2_{1,2}$, the classical Jacobi-Jordan algebra with 
	   $D(J')=\langle e_3\rangle$ is the representative $J'^2_{1,2}$, given 
	   by $[e_1,e_1]=e_3$. Twisting $J'^2_{1,2}$ by this Jordan block yields $\alpha([e_1,e_1])=a^2e_3$
	   while $[\alpha(e_1),\alpha(e_1)]=a^3e_3$; multiplicativity therefore forces
	   $a=1$, in which case the resulting bracket $[e_1,e_1]=e_3$ defines an HJJ
	   algebra, namely $J^1_3$.	   
\end{proof}

%======================

\end{document}